\documentclass[reqno,final,11pt]{amsart}
\usepackage{fancyhdr} 
\usepackage{color} 
\usepackage{hyperref} 
\usepackage{graphicx} 
\usepackage{geometry}
\usepackage[utf8]{inputenc}
\usepackage[T1]{fontenc}
\usepackage{verbatim}

\usepackage{tikz}
\usetikzlibrary{calc,shapes,backgrounds,arrows}
\usepackage{amssymb}
\usepackage{amsmath}
\usepackage{bbm}
 \usepackage{dsfont}
 
\definecolor{aleacolor}{rgb}{0.16,0.59,0.78}

\hypersetup{
breaklinks,
colorlinks=true,
linkcolor=aleacolor,
urlcolor=aleacolor,
citecolor=aleacolor}

\theoremstyle{plain}
\newtheorem{theorem}{Theorem}[section]                                          
                          
\newtheorem{lemma}[theorem]{Lemma}
\newtheorem{corollary}[theorem]{Corollary}

\theoremstyle{definition}
\newtheorem{definition}[theorem]{Definition}
\theoremstyle{remark}
\newtheorem{remark}[theorem]{Remark}
\newtheorem{example}[theorem]{Example}

\makeatletter \@addtoreset{equation}{section} \makeatother

\newcommand{\R}{\mathbb{R}}
\newcommand{\E}{\mathbb{E}}
\newcommand{\Indi}[1]{\mathbbm{1}_{#1}}
 
\begin{document}

\title[Short Title]{Rates on Yaglom's limit for Galton-Watson  processes  \\[3mm] in a varying environment}

\author{Natalia Cardona-Tob\'on}
\address{Institute for Mathematical Stochastics, Georg-August-Universität Göttingen. Goldschmidtstr. 7, C.P.  37077 Göttingen, Germany}
\email{natalia.cardonatobon@uni-goettingen.de} 
\urladdr{\href{https://sites.google.com/view/natalia-cardona-tobon/home-page}{https://sites.google.com/$\sim$natalia}} 

\author{Arturo Jaramillo}
\address{Centro de Investigación en Matemáticas. Calle Jalisco s/n. C.P. 36240, Guanajuato, México}
\email{jagil@cimat.mx} 
\urladdr{\href{https://www.cimat.mx/~jagil/app/dist/Arturo_Jaramillo_Gil.html}{https://www.cimat.mx/$\sim$jagil}}

\author{Sandra Palau}
\address{Universidad Nacional Autónoma de México. Circuito Escolar S/N, Ciudad Universitaria, C.P. 04510, Ciudad de México, México}
\email{sandra@sigma.iimas.unam.mx} 
\urladdr{\href{http://sigma.iimas.unam.mx/sandra/}{http://sigma.iimas.unam.mx/$\sim$sandra}}

\thanks{Arturo Jaramillo was supported by CONACYT, Mexico grant CB-2017-2018-A1-S-9764.}
\thanks{Sandra Palau was supported by grant  PAPIIT-UNAM no. IN103924.}

\begin{abstract}
 	A Galton-Watson process in a varying environment is a discrete time branching process where the offspring distributions vary among generations. It is known that in the critical case, these processes have a Yaglom limit, that is, a suitable normalization of the process conditioned on non-extinction converges in distribution to a standard exponential random variable. In this manuscript, we provide the rate of convergence of the Yaglom limit with respect to the Wasserstein metric.

\end{abstract}

\maketitle

\section{Introduction and  main results}\label{sec:intro}
The goal of this manuscript is to establish a quantitative comparison between the law of a suitably re-scaled Galton-Watson process with time dependence offspring distributions, conditioned on being positive, against a standard exponential distribution, measured with respect to the Wasserstein distance. 
The precise setting of this problem lies in the realm of Galton-Watson processes in a varying environment. 
Its central concept consists of a system of particles 
which independently produce descendants and the 
offspring distributions vary among generations. 

A Galton-Watson processes in a varying environment (GWVE for short)  is constructed in the following way: Consider a collection $Q=\{q_n:  n\geq 1\}$ of probability measures defined over a given space $(\Omega, \mathcal{F},\mathbb{P})$ and supported on $\mathbb{N}_0:=\{0,1,2,\dots \}$. The sequence $Q$ is called an environment.  As it is customary when studying discrete measures, we will interchangeably use the notation $q_n[\{k\}]$ and $q_n[k]$ for $k\in\mathbb{N}_0$. A  {\it Galton-Watson process} $Z = \{Z_n: n\geq 0\}$ in the {\it environment} $Q$ is defined recursively by 
\begin{align}\label{eq:Zndef}
	Z_0=1\quad \mbox{ and } \quad Z_n
	=\underset{i=1}{\overset{ Z_{n-1}}{\sum}}\chi_{i}^{(n)}, \qquad n\geq 1,
\end{align}
where  $\{\chi_{i}^{(n)}: i,n\geq 1\}$ is a sequence of independent random variables over $(\Omega, \mathcal{F},\mathbb{P})$, satisfying
$$\mathbb{P}[\chi_{i}^{(n)}=k]=q_n[k], \qquad \quad k\in \mathbb{N}_{0},\ i,\ n\geq 1.$$ 
The quantity $\chi_{i}^{(n)}$  denotes the offspring of the $i$-th individual in the $(n-1)$-th generation. The process $Z_n$ represents the total population size at generation $n$. It satisfies the Markov property with respect to its naturally induced filtration. Without further mention we
always require that $0<\mathbb{E}[\chi_{i}^{(n)}]<
\infty$  for all  $n\geq 1$.
The classical Galton-Watson process corresponds to  $q_{n}\equiv q$ for all $n\geq 1$, and in the sequel, will be referred to as the {\it constant environment regime}. 

Let us recall what is known in the constant environment regime. 
In order to exclude trivial cases, we assume in the following that 
\begin{equation}\label{eq:q0q1}
	q[0] + q[1] < 1 \qquad \text{and}\qquad q[k]\not =1 \qquad \text{for all}\qquad k \in \mathbb{N}_0,
\end{equation}
meaning that the probability of having offspring is strictly positive and that $q$ is not deterministic.
In this case, the mean  $\mu:=\E[\chi^{(1)}_1]$ is a good parameter in order to study the asymptotic behaviour of the process. More precisely,  the process has extinction almost surely if and only if $\mu\leq 1$. When $\mu>1$, (\textit{supercritical case}), in the event of survival, the process $Z_n$ goes to $\infty$ as $n\rightarrow \infty$. When $\mu<1$, (\textit{subcritical case}),  the process dies out with probability one and the distribution of $\{Z_n\mid Z_n>0\}$ converges to a proper distribution. 
The \textit{critical case}, when $\mu=1$, is in a sense the most interesting case because $Z_n\rightarrow 0$ as $n\rightarrow \infty$ but the conditional process $\{Z_n\mid Z_n>0\}$  is diverging to $\infty$. 
Then,  a normalization and a finite second moment of the offspring distribution is needed to make that the 
conditioned process converge to a non-degenerate limit. If $\mu=1$, $\sigma^2:=var(Z_1)<\infty$ and $\E[(Z_1)^3]<\infty$ then as $n\rightarrow \infty$,  the survival probability decays as 
\begin{equation}\label{eq: kolmogorov estimate}
	\mathbb{P}[Z_n>0]= \frac{2}{\sigma^2n}+O\left( \frac{{\log(n)}}{n^2}\right).
\end{equation}
Kolmogorov (1938) first proved   $\mathbb{P}[Z_n>0]\sim  2(\sigma^2n)^{-1}$ without giving explicitly the exact decay. {The proof with the explicit error bound is provided in \cite[Theorem 3]{Nagaev}; see also \cite[Theorem 4]{Gaviev1969} and \cite[Display between (5) and (6), page 2437]{vatutin1987}. Note that in the latter reference, there is a typographical error in the error bound, which was carried over into the published version of this article in ALEA. We would also like to point out that the original references are in Russian, but we have cited them here in English for convenience. Now,}
Yaglom  \cite{yaglom1947certain} proved   under a third moment assumption, that  $2Z_n(\sigma^2n)^{-1}$ conditioned on the event $\{Z_n>0\}$ converges in distribution to a standard exponential variable as $n\rightarrow \infty$.
\begin{theorem}[\textbf{Yaglom's limit for GW}]\label{thm:Yaglom}
	Let $\{Z_n:n\geq 0\}$ be a critical Galton-Watson process with $\sigma^2<\infty$. Then, 
	\begin{align*}
		\left\{\left.\frac{2}{\sigma^2n}Z_n\ \right| Z_{n}>0\right\}
		&\stackrel{(d)}{\longrightarrow}\mathbf{e}, \qquad \mbox{  as }\ n\rightarrow \infty,
	\end{align*}
	where $\mathbf{e}$ denotes a standard exponential random variable and ``$\stackrel{(d)}{\longrightarrow}$" 
	means convergence in distribution.
\end{theorem}
This result has several proofs. Yaglom proved it  by using the Laplace transform of the process.
Lyons et al. \cite{MR1349164} gave a 
proof using a characterisation of the exponential distribution via its size-biased distribution.  
They showed that the size-biased process can be related with  a  Galton-Watson tree with one distinguished genealogical line (the so-called \textit{spine}). 
Later on, Geiger characterised the exponential random variable by a distributional equation and  the author presented another proof of Yaglom's limit based on that equation (see  \cite{geiger1999elementary,geiger2000new}).  Additionally, \cite{ren20172}, developed  another proof using a two-spine decomposition technique, where the associated Galton-Watson tree has two distinguished genealogical lines.

More than 60 years after Yaglom proved Theorem \ref{thm:Yaglom} in the constant environment, \cite{pekoz2011} obtained  explicit error bounds. They used the Wasserstein distance to compare the distributions. For two measures $\mu$ and $\nu$ we define the \textit{Wasserstein distance}, $d_W$, as follows
\begin{align*}
	d_W  (\mu,\nu)
	&:=\sup_{f\in \mathcal{F}}\left|\int_{\R}f(x)\mu(\mathrm{d}x)-\int_{\R}f(x)\nu(\mathrm{d}x)\right|,
\end{align*}
where $\mathcal{F}=\{f:\R\rightarrow\R\ : f\mbox{ is Lipschitz and } \|f'\|\leq 1 \}$, and $f'$ is 
the derivative of $f$. For random variables $X$ and $Y$ 
with respective laws $\mu$ and $\nu$, we abuse notation and write $d_W(X,Y)$ in place of $d_W(\mu, \nu)$. Peköz and Röllin \cite{pekoz2011} studied the rate of convergence by using Stein's method, a collection of probabilistic techniques tailored for estimating distances by means of differential operators. The manuscript \cite{pekoz2011}, building upon previous work by \cite{MR1401468}, presents a sharp Stein's method machinery for exponential approximations with a perspective of equilibrium distribution. Subsequently, in conjunction with the relation of the size-biased process 
and the Galton-Watson process with one spine of  \cite{MR1349164}, the authors implement their methodology in the framework of branching processes and establish the following quantitative version of Yaglom's theorem. In \cite[Theorem 3.3]{pekoz2011}, the authors stated:

\textit { If $\{Z_n:n\geq 0\}$ is a critical Galton-Watson process with $\sigma^2<\infty$ and $\E[(Z_1)^3]<\infty$. Then, there exists a constant $C>0$ such that 
	\begin{equation}\label{eq:pekozthm}
		d_W\left(\left\{\left.\frac{2}{\sigma^2n}Z_n\ \right| Z_{n}>0\right\}, \mathbf{e}\right)
		\leq  C\ \frac{\log(n)}{n}.
\end{equation}}

Unfortunately, the proof of \eqref{eq:pekozthm} in  \cite{pekoz2011} has a minor mistake: they used \cite[Theorem 2.1]{pekoz2011},
which is only valid when the random variable of interest has mean one and thus it is not applicable to $\{2(\sigma^2n)^{-1}Z_n\mid Z_n>0\}$. {However, the bound still holds.}
This issue {in the proof} can be easily corrected; we propose two options for doing so. One is to change the normalization in order to have a mean one variable.  The second option is to extend \cite[Theorem 2.1]{pekoz2011} to random variables with finite mean $m$. We develop such extension in Theorem \ref{theo:pekoz}, where the distance $|m-1|$ becomes part of the bound.
Essentially the first bound follows from Theorem \ref{theo:pekoz} and \cite[Theorem 3.3]{pekoz2011}, together with  the fact that the mean of a critical Galton-Watson process satisfies 
\begin{equation}\label{eq:kolmesti}
	\E\left[\left.\frac{2}{\sigma^2n}Z_n\ \right| Z_{n}>0\right]= { \frac{2}{\sigma^2n}\frac{1}{\mathbb{P}[Z_n>0]}=\frac{1}{ 1+O\left( \frac{\log(n)}{n}\right)}=1+O\left( \frac{\log(n)}{n}\right),}
\end{equation}
which is implied by equation \eqref{eq: kolmogorov estimate}. On the other hand, the second bound is obtained if, instead, we use the mean one variable and make a slight modification of the proof of \cite[Theorem 3.3]{pekoz2011}. {The two bounds coincide, as stated in the following theorem, with the proof deferred to Section~\ref{proof constant}.}

\begin{theorem}\label{thm:PekozRoellin}
	Let $\{Z_n:n\geq 0\}$ be  a critical Galton-Watson process with $\sigma^2<\infty$ and $\E[(Z_1)^3]<\infty$. Then, there exists a constant $C>0$ which is independent of $n$ such that 
	\begin{align*}
		d_W\left(\left\{\left.\frac{2}{\sigma^2n}Z_n\ \right| Z_{n}>0\right\}, \mathbf{e}\right)
		&\leq  C\ \frac{{\log(n)}}{n},
	\end{align*}
	and
	\begin{align}\label{eq:dPeRo1}
		d_W\left(\left\{\left.\ \mathbb{P}[Z_n>0]\  Z_n\ \right| Z_{n}>0\right\}, \mathbf{e}\right)
		&\leq  C\ \frac{\log(n)}{n}.
	\end{align}
\end{theorem}
This bound might be strictly suboptimal, as in some instances can be improved  by a factor of the form $1/\log(n)$. This is illustrated in Example \ref{ex:linear}, where the offspring has linear fractional distribution. For this instance,
\begin{align*}
	d_W\left(\left\{\left.\ \mathbb{P}[Z_n>0]\  Z_n\ \right| Z_{n}>0\right\}, \mathbf{e}\right)
	&\leq  \frac{4}{2+\sigma^2 n}< C\ \frac{\log(n)}{n},
\end{align*}
for a constant $C$ independent of $n$.

The above phenomenology suggests that a similar behaviour should hold for Galton-Watson processes in a varying environment. We require some knowledge on the long-term behaviour of these processes. Let  $Q=\{q_n: n\geq 1\}$ be an environment and for every $n\geq 1$, denote by $f_n$ the  generating function associated with the reproduction law $q_n$, i.e.
$$ f_n(s):= \sum_{k=0}^{\infty} s^k q_n[k], 
\qquad 0 \leq s\leq 1.$$ 

Let $
\{Z_{n}: n\geq 0\}$ be a Galton-Watson process in   environment $Q$ defined as in \eqref{eq:Zndef}. By a recursive application of the branching property, we deduce that the generating function of \ $Z_n$ \ can be written in terms of the $f_i$'s as 
\begin{equation}\label{laplace Z}
	\mathbb{E}\big[s^{Z_n}\big]=f_1\circ\cdots\circ f_n(s), 
	\qquad 0 \leq s\leq 1, \ n\geq 1,
\end{equation}
where $``\circ"$ denotes the composition of functions. This expression gives an explicit description of the law of $Z_n$ in terms of the reproduction laws $\{q_k: 1\leq k\leq n\}$.
It guarantees that we can obtain formulas for moments of $Z_n$ in a standard manner.  More precisely, let $\mu_0:=1$, and for any $n\geq 1$, define 
\begin{equation}\label{eq:mu}
	\mu_n:=f_1'(1)\cdots f_n'(1), \qquad \nu_n:= \frac{f_n''(1)}{f_n'(1)^2},\qquad \mbox{and} \qquad\rho_{0,n}:=\sum_{k=0}^{n-1}\frac{\nu_{k+1}}{\mu_{k}}.
\end{equation}
Then, by taking derivatives with respect to the variable $s$ and then evaluating the resulting function at $s=1$,  we deduce that the mean and the normalized second factorial moment of $Z_n$ satisfy
\begin{equation*}
	\mathbb{E}\big[Z_n\big]=\mu_n \qquad \mbox{and} \qquad \frac{\mathbb{E}\big[Z_n(Z_n-1)\big]}{\mathbb{E}\big[Z_n\big]^2}=\rho_{0,n}.
\end{equation*}
The interested reader is referred to the monograph of Kersting and Vatutin \cite{kersting2017discrete} for a detailed presentation of the formulas above.

Galton-Watson processes in a varying environment may behave different from the constant environment regime. 
During a certain period, research on GWVE was temporarily affected due to the emergence of families exhibiting exotic properties.
For example, they may possesses multiple rates of growth, as was detected by \cite{macphee1983galton}. Moreover, D’Souza \cite{d1994rates} constructed another example where the GWVE has an infinite number of rates of growth. We would like to present this example here and  refer the reader to Section 4 in \cite{d1994rates} for further details. Let $\{a_n: n\ge 0\}$ be a strictly increasing sequence such that $a_0>0$ and $\lim_{n\to \infty}a_n:= a_\infty<\infty$. Choose another sequence $\{b_n: n\ge 0\}$  satisfying $0< b_0 <a_0 <b_1 < a_1< \dots $
and define 	\[\beta_n:= \sum_{i=1}^{\infty} \frac{a_i - a_{i-1}}{n^{b_i+1}}, \quad n\ge 1.\]
Now, consider $\{Z_n: n\geq 0\}$ a GWVE with generation functions given by
\[f_n(s)= \left(1-\frac{a_0}{n} -2 \beta_n\right)s + \left(\frac{a_0}{n} + \beta_n\right)s^2 + \sum_{i=1}^{\infty} \frac{a_i- a_{i-1}}{n^{b_i+1}}s^{n^{b_i}},\]
for $n \ge n_0$, where $n_0$ is large enough such that this defines a proper probability generating function, and $f_n(s)=s$ for $n< n_0$. According to  \cite[Theorem 6]{d1994rates}, for any $i\in \mathbb{N}$ the sequence $\{Z_n / n^i: n\ge 0\}$ converges to a finite, positive limit with non-zero probability.

The presence of these exotic properties gave the impression that understanding a generic behaviour for these processes was challenging. However,   Kersting \cite{kersting2020} established a
certain condition which excludes such exceptional phenomena: for every $\epsilon>0$, there is a finite constant $c_\epsilon$ such that for all  $n\ge 1$
\begin{equation}\label{cond:Kersting}\tag{$\star$}
	\mathbb{E}\left[\left(\chi^{(n)}\right)^2\mathbf{1}_{\{\chi^{(n)}>c_\epsilon(1+\mathbb{E}[\chi^{(n)}])\}}\right]\leq \epsilon \mathbb{E}\left[\left(\chi^{(n)}\right)^2\mathbf{1}_{\{\chi^{(n)}>2\}}\right]<\infty,
\end{equation} 
where $\chi^{(n)}\sim q_n$.  
We say that a GWVE is regular if it satisfies Condition  \eqref{cond:Kersting}. In what follows, we always consider regular GWVE.

Nonetheless, verifying directly the latter  condition  for many families of random variables can be cumbersome, so  instead as it is suggested in \cite{kersting2020} one may use the following mild third moment condition: there exists $ c>0$ such that    
\begin{equation}\label{cond:thirdmoment}\tag{$\star \star$}
	f_n'''(1)\leq c f_n''(1)(1+ f_n'(1)),  \mbox{ for any } n\geq 1.
\end{equation} 
According to  \cite[Proposition 2]{kersting2020},  it implies condition \eqref{cond:Kersting} and it is satisfied by the most common probability distributions, for instance binomial, geometric, hypergeometric, negative binomial  and Poisson distributions.

It turns out that under Condition \eqref{cond:Kersting}, the behaviour of a GWVE  is essentially dictated
by the sequences  $\{\mu_n: n\geq 0\}$ and $\{\rho_{0,n}: n\geq 0\}$. With them, a regular GWVE can be classified into four regimes. 
According with \cite[Theorem 1]{kersting2020}, a regular GWVE has almost sure extinction if and only if $\mu_n\rightarrow 0$ or $\rho_{0,n}\rightarrow \infty$ as $n\rightarrow \infty$. Moreover, under these conditions,  $\mathbb{E}\big[Z_n \mid   Z_n>0\big] \to \infty$ as $n\rightarrow \infty$ if and only if $\mu_n\rho_{0,n}\to \infty$ as $n\rightarrow \infty$, see \cite[Theorem 4]{kersting2020}.  Therefore, in comparison with   the constant environment regime, 
it is natural to give the following definitions.  A regular GWVE is \textit{critical} if and only if 
\begin{equation}\label{eq:critical}
	\rho_{0,n} \to \infty \qquad \text{ and } \qquad   \mu_n \rho_{0,n} \to \infty,\qquad  \text{as} \quad  n\to \infty.
\end{equation}
A regular GWVE is \textit{subcritical} if  it has almost sure extinction and $\liminf_{n\rightarrow \infty}\mu_n \rho_{0,n}<\infty$. If the extinction probability is less than one, the \textit{supercritical regime}  is in the case $\mathbb{E}[Z_n]\rightarrow \infty$, and the \textit{asymptotically degenerate regime} otherwise. The latter is a new regime where process may freeze in a positive state.
We refer to \cite{kersting2020} for the details.

In the sequel, we work exclusively with (regular and) critical Galton-Watson processes and focus on the study of the asymptotic behaviour of $Z_{n}$ conditioned on being positive. As it was noted by  \cite{kersting2020}, for a critical GWVE the so-called Yaglom's limit exists: 
\begin{theorem}[\textbf{Yaglom's limit for GWVE}]\label{eq:nonquantitativeYaglom}
	Let ~$\{Z_n: n\geq 0\}$ be a (regular) critical GWVE.  Then, there is a sequence $\{b_n: n\geq 0\}$ of positive numbers such that 
	\begin{align}\label{eq Yaglom}
		\left\{\left.\frac{1}{b_n}Z_n\ \right| Z_{n}>0\right\}
		&\stackrel{(d)}{\longrightarrow}\mathbf{e}, \qquad \mbox{  as }\ n\rightarrow \infty,
	\end{align}
	where $\mathbf{e}$ is a standard exponential random variable. The sequence $\{b_n: n\geq 0\}$ may be set as $\{\E[Z_n\mid Z_n>0]: n\geq 0\}$ or  $\{\mu_n\rho_{0,n}/2: n\geq 0\}$, since $\E[Z_n\mid Z_n>0]\sim  \mu_n\rho_{0,n}/2$ as $n\rightarrow \infty$.
\end{theorem}

For the sequence $\{\mu_n\rho_{0,n}/2: n\geq 0\}$, the previous  limit was obtained by Jagers \cite{jagers1974galton} under extra assumptions. Afterwards, Bhattacharya and Perlman \cite{bhattacharya2017time} obtained the same result with  weaker assumptions than Jagers (but stronger than \eqref{cond:thirdmoment}). Kersting \cite{kersting2020}  provided yet another proof under condition \eqref{cond:Kersting}.  Moreover, he showed the convergence in distribution for both sequences $\{\mu_n\rho_{0,n}/2: n\geq 0\}$ and $\left\{\E[Z_n\mid Z_n>0]: n\geq 0\right\}$.  An extension in the presence of immigration and the same setting as Kersting's has been established in \cite{gonzalez2019branching}. All these authors established the exponential convergence using an analytical approach.  The asymptotic distributional behaviour of the variables $Z_{n}/b_n$ appearing in \eqref{eq Yaglom} with $\{b_n: n\geq 0\}$ given by $b_n=\mu_n\rho_{0,n}/2$ was obtained by   \cite{cardona2021yaglom}, under condition \eqref{cond:thirdmoment}, through an 
approach based on a two-spine decomposition argument.

The goal of this manuscript is to study the rate of convergence of the previous limit under the Wasserstein distance. We want to extend Theorem  \ref{thm:PekozRoellin} to Galton-Watson processes in a varying environment. Observe that Theorem  \ref{thm:PekozRoellin} requires a third moment condition in the offspring distribution. 
In our case, since the environment is varying, we need to control the third moment of every $q_k$.
Since, in addition we are  asking for  the  regular condition \eqref{cond:Kersting} and the critical condition  \eqref{eq:critical}, then it is natural to just ask conditions \eqref{cond:thirdmoment} and  \eqref{eq:critical}.

In the constant environment regime, consider $b_n= \E[Z_n\mid Z_n>0]=\mathbb{P}[Z_n>0]^{-1}$,  $\mu_n=1$ and  $\rho_{0,n}=\mu_n \rho_{0,n}=\sigma^2n$, for any $n\geq 1$.  Then, bound \eqref{eq:dPeRo1} in Theorem \ref{thm:PekozRoellin} can be written as 
\begin{align}\label{eq:quantitativePekRoe}
	d_W\left(\left\{\left.\frac{1}{b_n}Z_n\ \right| Z_{n}>0\right\},\mathbf{e}\right)
	&\leq C\frac{\log(n)}{n}=\sigma^2 C\left(\frac{\log(\sigma^2n)+\log(\sigma^{-2} )}{\sigma^2 n}\right),
\end{align}
for some constant $C>0$ independent of $n$.  In this regime, we have that the sequence $\rho_{0,n}$ coincides with $\mu_n \rho_{0,n}$ and is of the order $n$. In particular, \eqref{eq:quantitativePekRoe} implies that the Wasserstein distance in the left-hand side of \eqref{eq:quantitativePekRoe} is bounded by a function of $\rho_{0,n}$ and $\mu_n \rho_{0,n}$. In a general varying environment, $\rho_{0,n}$ and $\mu_n \rho_{0,n}$ do not necessarily coincide. Furthermore, as demonstrated in the examples from Section \ref{sec:examples}, these quantities can grow in such a way that their quotient could converge to zero in some instances or infinity in others. Keeping in mind the fact of the discussion above, we seek pursuit a bound of the form $\psi(\mu_{n}\rho_{0,n})+\phi(\rho_{0,n})$ for functions $\psi,\phi:\R\rightarrow\R$ continuously vanishing at infinity. 
In the following theorems, the results stated exhibit a tradeoff between generality on the environment $Q$ and simplicity in the  form of the resulting bounds. Specifically, if a bound holds for a wide range of environments $Q$, it becomes challenging to give a simple expression for it. Reciprocally, an accurate description of  decay is available only for a subgroup of critical GWVE. This tradeoff may be inherent to the approach we are following, and we have not discovered any evidence suggesting that the same phenomena will emerge in different perspectives on the problem. We will start with the general case.
We are going to analyse the Wasserstein distance between  $\{b_n^{-1} Z_n \mid Z_n>0\}$  and the exponential distribution. For $\{b_n: n\geq 0\}$, we can use the sequences $\left\{\E[Z_n\mid Z_n>0]: n\geq 0\right\}$ or  $\{\mu_n\rho_{0,n}/2: n\geq 0\}$.  We decide to use the first sequence because we obtain a mean one  process. The proof is based on Theorem \ref{theo:pekoz}, which can be applied to random variables with finite mean. If one wants to use the second sequence, there will be an extra term that comes  from $| 2(\mu_n\rho_{0,n})^{-1}\E[Z_n\mid Z_n>0]-1|$.

\begin{theorem}\label{theo:speedYaglom}
	Let ~$\{Z_n: n\geq 0\}$ be a critical GWVE that satisfies condition \eqref{cond:thirdmoment} and define $b_n:=\E[Z_n\mid Z_n>0]$, $n\geq 0$. Then,
	\begin{align*}
		d_W\left(\left\{\left.\frac{1}{b_n}Z_n\ \right| Z_{n}>0\right\},\mathbf{e}\right)
		&\leq C\left(  \frac{1}{\mu_n \rho_{0,n}} + \frac{r_n}{\rho_{0,n}}\right),
	\end{align*}
	where $\mathbf{e}$ is a standard exponentially distributed random variable, $C>0$ is a constant which is independent of $n$  and 
	\begin{equation}\label{eq:rn}
		r_n:=\sum_{j=1}^{n-1}\frac{\nu_j}{\mu_{j-1}} \frac{\left(1+f_j'(1)\right)}{\mu_j(\rho_{0,n}- \rho_{0,j})} + \frac{ \nu_n}{ \mu_{n-1}} \left(1+f_n'(1)\right), \qquad n\ge 2.
	\end{equation}
\end{theorem}

It seems difficult to directly deduce rates of convergence from \eqref{eq:rn},  as the term $n$ is both in the upper limit of the sum, and in the denominator $\rho_{0,n}-\rho_{0,j}$. The following estimates may be more useful. For stating the result in its most general version, we define the following bound on the first moment of the offspring, 
\begin{align*}
	\mathfrak{M}_{n}
	:=\sup_{1\leq k\leq n}f_k^{\prime}(1), \quad\text{for}\quad  n\geq 1,
\end{align*} 
as well as the following sum
\begin{align}\label{eq:sn}
	s_n
	&= \sum_{k=2}^{n-1} \left(\log(\rho_{0,k} \mu_k ) +\log\left(f_k'(1)\right)\right)
	\left|\frac{1}{\mu_{k-2}} - \frac{1}{\mu_{k}}\right|, \quad \text{for}\quad n\ge3.
\end{align}
In a variety of examples,  these two terms are bounded uniformly over $n$, instance in which they might be controlled from above by fixed constants. 
\begin{theorem}\label{theo:rnestimation}
	Let $\{Z_n: n\geq 0\}$ be a critical GWVE that satisfies condition \eqref{cond:thirdmoment} and define $b_n:=\E[Z_n\mid Z_n>0]$, $n\geq 0$. Assume that there exists $a>0$ such that $f_{n}''(1)\ge a$ for all $n\ge 1$. Then
	\begin{align}\label{ineq:mainsecond}
		&d_W\left(\left\{\left.\frac{1}{b_n}Z_n\ \right| Z_{n}>0\right\},\mathbf{e}\right)
		\leq  C (1+	\mathfrak{M}_{n})^5\left(   \frac{\log(\mu_n \rho_{0,n})+\log(f'_n(1))}{\mu_n\rho_{0,n}} + \frac{{s}_n}{\rho_{0,n}}\right),
	\end{align}	
	where $\mathbf{e}$ is a standard exponentially distributed random variable and $C>0$ is a constant which is independent of $n$.
\end{theorem}

\begin{remark}\label{rem:remarkq0q1}
	A sufficient condition for the sequence  $\{f_n''(1): n\ge 1\}$ to be bounded away from zero is the following
	\begin{equation*}
		\limsup_{n\rightarrow\infty} \big(q_n[0] + q_n[1]\big)<1. 
	\end{equation*} 
	The latter condition tells us that we are excluding trivial situations as we already mentioned for the analogue condition \eqref{eq:q0q1} in the constant environment case. 
	This condition implies that the sequence $\{q_n[0] + q_n[1]: n\ge 0\}$ is bounded away from one. Roughly speaking, this condition prevents the scenario where the population stagnates in a state where every individual only produces zero or one offspring in the next generation.
	Bhattacharya and Perlman asked for the latter condition in their hypothesis for Yaglom’s Theorem, see (H2) in \cite{bhattacharya2017time}.
\end{remark}

In particular, if the mean and the second factorial moment of the offspring distribution are bounded away from zero and infinity and  if we \textit{also} assume $\lim_{n\to \infty} {s}_n <\infty$, we can obtain an explicit bound only depending on $\rho_{0,n}$ and $\mu_n\rho_{0,n}$. More precisely, we have the following direct consequence.

\begin{corollary}\label{theo:speedYaglomzero}
	Let ~$\{Z_n: n\geq 0\}$ be a critical GWVE that satisfies condition \eqref{cond:thirdmoment}  and define $b_n:=\E[Z_n\mid Z_n>0]$, $n\geq 0$. 
	Assume that there exists $0<a\leq A<\infty$ such that $a\leq f_n'(1),\ f_n^{\prime\prime}(1)\leq A$ for all $n\geq 1$,
	then
	\begin{align}
		d_W\left(\left\{\left.\frac{1}{b_n}Z_n\ \right| Z_{n}>0\right\},\mathbf{e}\right)
		&\leq C\left(  \frac{\log(\mu_n \rho_{0,n})}{\mu_n \rho_{0,n}} + \frac{s_n}{\rho_{0,n}}\right).
	\end{align}
	where $\mathbf{e}$ is a standard exponentially distributed random variable and $C>0$ is a constant which is independent of $n$. In particular, if we also assume ~$\lim_{n\to \infty} {s}_n<\infty$, we have
	\begin{align}\label{eq:mainbound1}
		d_W\left(\left\{\left.\frac{1}{b_n}Z_n\ \right| Z_{n}>0\right\},\mathbf{e}\right)
		&\leq C\left(  \frac{\log(\mu_n \rho_{0,n})}{\mu_n \rho_{0,n}} + \frac{1}{\rho_{0,n}}\right).
	\end{align}
\end{corollary}

\begin{remark}
	In the constant environment regime, the hypotheses of Corollary \ref{theo:speedYaglomzero} hold. Indeed, in the critical case, $f_n'(1)=1$, $\mu_n=1$, $\nu_n=\sigma^2$, $\rho_{0,n}=\sigma^2 n$, $b_n=\mathbb{P}[Z_n>0]^{-1}$ and $s_n=0$, for all $n\geq 3$. Moreover, condition \eqref{cond:thirdmoment} is equivalent to $\E[(Z_1)^3]<\infty$.
	Then, the bound \eqref{eq:mainbound1} is exactly  inequality \eqref{eq:quantitativePekRoe} and it implies Theorem \ref{thm:PekozRoellin} as a consequence.
\end{remark}

%

\noindent  The remainder of this paper is organized as follows. In Section \ref{sec:examples} we discuss a variety of different examples. In Section \ref{sec:preliminaries} we  
present Stein’s method for an exponential distribution and we relate it with the size biased distribution of $\{Z_n:n\geq 0\}$. Section \ref{sec:proofs}   contains the proofs  of our main theorems.

\section{Examples}\label{sec:examples}
In the next examples we discuss some cases where we can find explicitly the rate of convergence in Yaglom’s Theorem. In the first three, we use Theorem  \ref{theo:rnestimation} or  Corollary \ref{theo:speedYaglomzero}  to give  the bounds. Example \ref{ex:linear} shows that this rate  could be suboptimal. In what follows, $C_1, C_2, \dots$ are strictly positive constants independents of $n$ that change between examples.

\begin{example}[\textbf{Symmetric distributions}]   
	For each $n\geq 1$ and $a\in (0,1)$, let 
	\begin{equation*}
		q_n[0] = q_n[2]= \frac{1}{2n^{a}}\qquad \text{and} \qquad q_n[1]= 1-\frac{1}{n^a}.
	\end{equation*}
	Then, ~$f_n'(1) =1$, ~$f_n''(1)= 1/n^a$, and $f_n'''(1)= 0$ for every $n\geq 1$. The conditions of Theorem \ref{theo:rnestimation} do not hold due to the fact that $f_n''(1)$ converges to zero, so we will make use of Theorem \ref{theo:speedYaglom} instead. Note that,
	$$\mu_n=1,\qquad \nu_n= \frac{1}{n^a}\qquad  \mbox{and}  \qquad \rho_{0,n}= \sum_{k=0}^{n-1}\frac{1}{(k+1)^{a}}, \qquad   \mbox{for }\quad  n\ge 1.$$
	By an integral comparison, one can 
	check that  $\rho_{0,n}\sim n^{1-a}/(1-a)$  as $n$ tends to infinity and we are in the critical regime. By the same integral comparison, we also have as $n\to \infty$
	\[\rho_{0,n} - \rho_{0,j} = \sum_{k=j}^{n-1} \frac{1}{(k+1)^{a}}\geq  
	\int_{j}^{n-1} \frac{1}{(x+1)^{a}}dx
	= \frac{1}{1-a}(n^{1-a} - (j+1)^{1-a}).\]
	Moreover, by a Taylor expansion, $n^{1-a}-(j+1)^{1-a}\geq (1-a)n^{-a}(n-j-1),$ so we deduce that there exists a constant $C_1>0$ such that 
	\begin{align*}
		\sum_{j=1}^{n-1} \frac{1}{j^a(n^{1-a} - (j+1)^{1-a})}
		&\leq C_1n^{a}\sum_{j=1}^{n-1} \frac{1}{j^{a}(n-j)}  
		\leq C_2n^{a}\left(\sum_{j=1}^{n/2} \frac{1}{j^{a}(n-j)}+\sum_{j=n/2}^{n-1} \frac{1}{j^{a}(n-j)}\right)\\ 
		&\leq C_3n^{a}\left(\frac{2}{n}\sum_{j=1}^{n/2} \frac{1}{j^{a}}+ \frac{2^{a}}{n^{a}}\sum_{j=n/2}^{n-1} \frac{1}{n-j}\right)\\
		&= C_4\left(\frac{2}{n}\sum_{j=1}^{n/2} \frac{1}{(j/n)^{a}}+ \frac{2^{a}}{n^{a}}\sum_{j=1}^{n/2} \frac{1}{j}\right).
	\end{align*}
	The second term in the right is bounded by a constant multiple of $\log(n)/n^{a}$, while the first one converges towards 
	$2\int_{0}^{1/2}x^{-a}dx$, which is finite due to the condition $a\in(0,1)$. From this point, we conclude that there exists a constant $K>0$ such that 
	\begin{align*}
		\sum_{j=1}^{n-1} \frac{1}{j^a(n^{1-a} - (j+1)^{1-a})}
		&\leq K.
	\end{align*}
	From here, it easily follows that $r_{n}$ is bounded. Hence, appealing to Theorem \ref{theo:speedYaglom}, 
	\begin{equation*}
		d_W\left(\left\{\left.\frac{1}{b_n}Z_n\ \right| Z_{n}>0\right\},\mathbf{e}\right)
		\leq C_5 \frac{1}{n^{a}}.
	\end{equation*}
\end{example}

\begin{example}[\textbf{Poisson distributions with $\mu_n$ increasing linearly}] Let $Q=\{q_n:  n\geq 1\}$ be a sequence
	of Poisson distributions with parameters  $\lambda_n=f_n'(1)$ for $n\geq 1$.  Recall that for a Poisson distribution $f''_n(1)= f'_n(1)^2$ and $f'''_n(1)=f'_n(1)^3\leq f'_n(1)^2 (1+f'_n(1))$  for all $n\geq 1$. This implies that  \eqref{cond:thirdmoment} is satisfied and 
	$$\mu_n=\lambda_1\cdots \lambda_n, \qquad \nu_n=1,\qquad  \text{and} \qquad \rho_{0,n} = \sum_{k=0}^{n-1}\frac{1}{\mu_{k}},\qquad n\ge1.$$
	Let us consider the special case when $$ \lambda_1=1\qquad \text{and}\qquad \lambda_n = \frac{n}{n-1}\quad  \text{for}\quad  n>1.$$
	It follows that ~$\mu_n = n$ for $n\geq 1$ and ~$\rho_{0,n} \sim \log (n)$ as ~$n\to \infty$. Thus, the criticality condition \eqref{eq:critical} is satisfied, and we are in a critical region close to supercritical behaviour.  Observe that ~$1\leq \lambda_n\leq 2$ for any $n\ge 1$, so that $f_{n}^{\prime\prime}(1)=\lambda_n^2$ lies within the interval $[1,4]$ for all $n\geq 1$. Moreover,  
	\begin{equation*}
		\begin{split}
			s_n\leq C_1\sum_{k=3}^{\infty} \frac{1}{k(k-2)}+\sum_{k=2}^{\infty} \log(k \rho_{0,k} ) \left|\frac{1}{{k-2}}-\frac{1}{k}\right| \  \leq  C_2\sum_{k=3}^{\infty}  \frac{\log(k \log (k))}{k(k-2)},
		\end{split}
	\end{equation*}
	
	Note that the right-hand side sequence converges, so the conditions of the second part of Corollary~\ref{theo:speedYaglomzero} hold.
	Hence, appealing to Corollary \ref{theo:speedYaglomzero}, we deduce that
	\[
	d_W\left(\left\{\left.\frac{1}{b_n}Z_n\ \right| Z_{n}>0\right\},\mathbf{e}\right)
	\leq C_3 \left(   \frac{\log(n \log(n))}{n \log(n)} + \frac{1}{\log(n)}\right) \leq C_4 \frac{1}{\log(n)}.
	\]
	It needs to be pointed out that, the latter rate is slower than { that of } a constant environment with offspring distribution given by a Poisson with parameter 1. The reason for this is that the normalized second factorial moment in a constant environment (equivalently the variance) grows faster than in a varying environment as $n$ tends to infinity. In other words, in a constant environment we have that $\rho_{0,n}= var(Z_n)=\sigma^2 n$, while in this varying environment $\rho_{0,n} \sim \log(n)$ as $n\to \infty$.
\end{example}

\begin{example}[\textbf{Poisson distributions with $\mu_n$ decreasing at an exponential rate}] Let $Q=\{q_n:  n\geq 1\}$ be a sequence
	of Poisson distributions with parameters
	\begin{equation*}
		\lambda_1 = \exp(-1), \qquad \text{and}\qquad \lambda_n= \frac{\exp(-\sqrt{n})}{\exp(-\sqrt{n-1})}, \quad \text{for}\quad  n>1.
	\end{equation*}
	Observe that  $\lim_{n\to \infty}\lambda_{n}=1$, thus implying that  $f_{n}^{\prime}(1),\, f_{n}^{\prime\prime}(1)$ are bounded away from zero and infinity. 
	Moreover  ~$\mu_n = \exp(-\sqrt{n})$ for $n\geq 0$ and by using L’Hopital’s rule we get
	\begin{equation*}
		\int_{0}^n \exp(\sqrt{x}) \mathrm{d} x \sim 2 \sqrt{n} \exp(\sqrt{n}), \quad \text{as} \quad n\to \infty. 
	\end{equation*}
	Thus ~$\rho_{0,n} \sim  2 \sqrt{n} \exp(\sqrt{n})$ as ~$n\to \infty$. Therefore, the  criticality condition \eqref{eq:critical} is satisfied,  and we are in a critical  region close to the subcritical behaviour. We observe that
	\begin{equation*}
		\begin{split}
			s_n &\leq C_1  \sum_{k=2}^{n-1} (\log(\rho_{0,k} \mu_k)+1) \big(  e^{\sqrt{k}} - e^{\sqrt{k-2}} \big)\leq C_2\sum_{k=2}^{n-1} \log(2\sqrt{k}) \big( e^{\sqrt{k}} - e^{\sqrt{k-2}} \big) \\
			&\leq C_3\log(2\sqrt{n-1}) e^{\sqrt{n-1}} \le  C_4 \log(2\sqrt{n}) \exp(\sqrt{n}),
		\end{split}
	\end{equation*}
	Therefore, from first part of the Corollary \ref{theo:speedYaglomzero} we obtain that 
	\begin{equation*}
		\begin{split}
			d_W\left(\left\{\left.\frac{1}{b_n}Z_n\ \right| Z_{n}>0\right\},\mathbf{e}\right)
			&\leq C_5 \left(   \frac{\log(\sqrt{n})}{\sqrt{n}} + \frac{\log(2\sqrt{n}) \exp(\sqrt{n})}{\sqrt{n}\exp(\sqrt{n})}\right)\leq C_6 \frac{\log(\sqrt{n})}{\sqrt{n}} 
		\end{split}
	\end{equation*}
\end{example}

\begin{example}[\textbf{Linear fractional distributions}]\label{ex:linear}
	For each $n\geq 1$, we assume that $q_n$ is a linear fractional distribution, i.e. there exist  ~$p_n\in (0,1)$ and ~$a_n\in (0,1]$ such that 
	$$q_n(0) = 1-a_n \qquad \mbox{and } \qquad q_n (k) = a_n p_n (1-p_n)^{k-1}, \qquad k\geq 1. $$
	In this case, 
	$$
	f_n(s)=1- \frac{a_n(1-s)}{1-(1-p_n)s}, \quad s\in [0,1], \qquad f'_n(1) = \frac{a_n}{p_n},\ \qquad \nu_n =\frac{f_n''(1)}{f_n'(1)^2}= \frac{2(1-p_n)}{a_n}.
	$$
	
	According with \cite[Chapter 1]{kersting2017discrete}, $Z_n$ is again linear fractional with mean and normalized second factorial moment as follows 
	$$\mu_n = \frac{a_1\cdots a_n}{p_1\cdots p_n} \qquad \text{and}\qquad \rho_{0,n} = 2\sum_{k=0}^{n-1}\frac{p_1\dots p_k (1-p_{k+1})}{a_1\dots a_{k+1}}.$$
	Furthermore, ~$Z_n$ conditioned on $\{Z_n>0\}$ has a geometric distribution, i.e. for $k\geq 1$ 
	$$\mathbb{P}\big(Z_n = k \mid Z_n>0\big) = \widehat{p}_n (1-\widehat{p}_n)^{k-1}, \qquad \text{where} \qquad  \widehat{p}_n = \frac{2}{2 + \mu_n \rho_{0,n}}.$$
	Denote by $G_n$ this geometric distribution. Observe  that $b_n=\E[Z_n\mid Z_n>0]=(2 + \mu_n \rho_{0,n})/2$. 
	Assume that the GWVE is critical, that is $\mu_n$ and $\rho_{0,n}$ satisfy the criticality conditions  \eqref{cond:thirdmoment} and \eqref{eq:critical}. 	Thus, appealing to Theorem \ref{theo:pekoz}, we deduce 
	\begin{equation*}
		\begin{split}
			d_W\left(\left\{\left.\frac{2}{2 + \mu_n \rho_{0,n}}Z_n\ \right| Z_{n}>0\right\},\mathbf{e}\right) &  \leq 2\mathbb{E}\left[ \left|\frac{2G_n}{2 + \mu_n \rho_{0,n}}- \left(\frac{2 G_n}{2 + \mu_n \rho_{0,n}}\right)^e\right|\right] .
		\end{split}
	\end{equation*}
	According to \cite[Example 5.9]{MR2861132}, if $G_n$ is a geometric distribution then,  $G_n-U$ has  the equilibrium distribution of $G_n$, where $U\sim \text{Unif}[0,1]$ is  independent of $G_n$. Together with  the fact that 	 ~$(c\, G_n)^e \overset{(d)}{=} c\, G_n^e$  for any constant $c$, we deduce
	\begin{equation*}
		\begin{split}
			d_W\left(\left\{\left.\frac{2}{2 + \mu_n \rho_{0,n}}Z_n\ \right| Z_{n}>0\right\},\mathbf{e}\right)\leq    \frac{4}{2+\mu_n \rho_{0,n}}  \mathbb{E}[U] = \frac{4}{2+\mu_n \rho_{0,n}}.
		\end{split}
	\end{equation*} 
\end{example}

\section{Preliminary results} \label{sec:preliminaries}
Peköz and Röllin \cite{pekoz2011} obtained the explicit error bounds discussed in equation \eqref{eq:pekozthm} by an implementation of their abstract results from Stein's method in the context of branching processes. This was carried through the construction of some  couplings of the size-biased distribution of $Z_{n}$, which in turn incarnated couplings for the equilibrium distribution. The enhancement of these ideas with the major advances in the understanding of the size-biased distribution of a GWVE,  naturally suggests that their methodology has the potential of being successfully implemented in the varying environment regime. Throughout the rest of this paper, we will show that this is indeed possible.

First, we present some preliminaries on Stein's method for the exponential distribution and extend the results of Peköz and Röllin to distributions with finite mean. After that, we give the proof of Theorem \ref{thm:PekozRoellin}. Then, we present the size-biased distribution of $Z_n$ and we relate it  with a GWVE tree with one spine. We also provide a brief review of the shape function associated with a probability generating function.

\subsection{Stein's method}\label{sec:stein}
In this subsection we develop Stein’s method for bounding the Wasserstein distance between an exponential distribution and another variable.  We refer to the survey \cite{MR2861132} to explore the techniques of Stein’s method for distributional approximation involving normal, Poisson, exponential, and geometric distributions. Recall that  the Wasserstein distance, $d_W$, is defined as
\begin{align*}
	d_W  (\mu,\nu)
	&:=\sup_{f\in \mathcal{F}}\left|\int_{\R}f(x)\mu(\mathrm{d}x)-\int_{\R}f(x)\nu(\mathrm{d}x)\right|,
\end{align*}
where $\mathcal{F}=\{f:\R\rightarrow\R\ : f\mbox{ is Lipschitz and } \|f'\|\leq 1 \}$. 
In the field of distributional approximation, the problem of finding sharp bounds for quantities of the form $d_W(\mu,\nu)$, for probability measures $\mu$ and $\nu$, had its first developments within a perspective of Fourier analysis arguments. This approach followed the heuristic that the Fourier inversion formula could transfer information regarding the characteristic functions of $\mu$ and $\nu$, to information about the action of test functions over the distributions, ultimately leading to bounds for $d_W(\mu,\nu)$. This method 
remains remarkably useful nowadays in a variety of situations. However, the distributional approximation problem witnessed a major breakthrough with the paper \cite{MR0402873}, where Charles Stein proposed an alternative methodology based on the idea of expressing the difference between the actions of a test function $f$ over  $\mu$ and $\nu$ in the form 
\begin{align*}
	\int_{\R}f(x)\mu(dx) -\int_{\R}f(x)\nu(dx)
	&=\int_{\R}\mathcal{S}[f](x)\mu(dx),
\end{align*} 
where $\mathcal{S}[f]$ is a differential operator with the property that its annihilation over a suitable domain of test functions under the action of $\mu$, is equivalent to the identity $\mu=\nu$. The original analysis from Stein focuses exclusively on the case where $\nu$ is the standard Gaussian distribution. However, many improvements and developments have since been made, allowing the field's community to develop tools for incorporating other choices for $\nu$. In this direction, we would like to highlight the work on the exponential distribution by \cite{pekoz2011}, which serves as one of the main components of our proofs. 
We take as our starting point the Stein characterization of the exponential distribution, as given by the following lemma. For the proof, refer to, for example, \cite[Lemma 5.2]{MR2861132}.
\begin{lemma}
	A random variable $X$ has standard exponential distribution if and only if for every continuously differentiable test function $h:\R\rightarrow\R$ with bounded derivative, 
	\begin{align*}
		\E[\mathcal{A}[h](X)]
		&=0,
	\end{align*}
	where $\mathcal{A}$ denotes the operator 
	\begin{align*}
		\mathcal{A}[h](x)
		&:=h^{\prime}(x)-h(x)+h(0).
	\end{align*}
\end{lemma}

Let $\mathbf{e}$ be a standard exponential random variable, and let $X$ be a random variable, both defined on the same probability space. This lemma, suggests implementing Stein's method by finding a solution $h_f$ (depending on $f$) of the following equation
\begin{align}\label{eq:steinsol}
	f(x) - \mathbb{E}[f(\mathbf{e})]= h_f'(x)- h_f(x) + h_f(0). 
\end{align}
It is straightforward to prove that the previous equation has a unique solution that satisfies $h_f(0)=0$ (see \cite[Lemma 5.3]{MR2861132}).  By applying \eqref{eq:steinsol} to $X$,  we obtain 
\begin{align*}
	\E[f(X)]-\E[f(\mathbf{e})]
	&=\E[\mathcal{A}[h_f](X)].
\end{align*}

Therefore,  the problem is reduced to bound the quantity $\E[\mathcal{A}[h_f](X)]$. The first step to obtain  a sharp estimation on this quantity requires some knowledge on the regularity of the solution $h_f$ to \eqref{eq:steinsol}, we refer the reader to \cite[Lemma 4.1]{pekoz2011} or \cite[Lemma 5.3]{MR2861132}. With that estimation in hand, they obtained the following result, see \cite[Theorem 5.4]{MR2861132}.

\begin{theorem}\label{theo:5.4}
	Let $X$ be a non-negative random variable with $\mathbb{E}[X]<\infty$ and $\mathbf{e}$ be a standard exponential random variable.  Then,
	\begin{equation*}
		d_W  (X,\mathbf{e})
		\leq \sup_{f\in \mathcal{F}_W}\left|\mathbb{E}[f'(X)-f(X)]\right|,
	\end{equation*}	
	where $\mathcal{F}_W=\{f:\R\rightarrow \R \ : f(0)=0,\  \|f'\|\leq 1, \mbox{ and } \|f''\|\leq 2\}$.
\end{theorem}

Considering the shape of the error, it becomes necessary to identify a structure that facilitates the comparison between $\mathbb{E}[f(X)]$ and $\mathbb{E}[f'(X)]$ for some $f$, thereby obtaining a bound for the distance between $X$ and $\mathbf{e}$. We introduce the following definition.

\begin{definition}
	Let $X$ be a non-negative random variable with a positive finite mean. We say that  $X^e$ has the \textit{equilibrium distribution} with respect to $X$, if for every Lipschitz function $f:\R\rightarrow\R$, 
	\begin{equation*}
		\mathbb{E}[f^{\prime}(X^e)]=\frac{	\mathbb{E}\big[f(X)-f(0)\big]}{\mathbb{E}[X]}.
	\end{equation*}
\end{definition}

It is not clear whether an equilibrium distribution for $X$ exists. However, $X^e$ can 
be related to a uniform random variable and the size-biased distribution of $X$. We recall that a random variable $\dot{X}$ has the \textit{size-biased distribution} of $X$ if, for all bounded or Lipschitz 
functions $f$,
\begin{equation*}
	\mathbb{E}\big[f(\dot{X})\big] = \frac{\mathbb{E}[Xf(X)]}{ \mathbb{E}[X]}.
\end{equation*}
The existence of a size-biased distribution is directly guaranteed by the Radon-Nikodym Theorem. Let $U$ be a uniformly distributed random variable over $[0,1]$ and independent of $\dot{X}$. For every $x\in \R$, the Fundamental Theorem of Calculus allows us to express $f(x)-f(0)=\E[xf^{\prime}(xU)]$. Then, by independence, 
\begin{align*}
	\E[f(X)-f(0)]
	&=\E[Xf^{\prime}(XU)]=\E[X]\E[f^{\prime}(\dot{X}U)].
\end{align*} 
This proves the following lemma.
\begin{lemma}\label{lem:sizebiastoequilibrium}
	Let $\dot{X}$ be the size-biased distribution of $X$ and $U$ a Uniform $(0,1)$ random variable independent of $\dot{X}$. Then, ~$U\dot{X}$ has the equilibrium distribution of $X$.
\end{lemma}

Now that we know the existence of the equilibrium distribution, we are ready to bound the Wasserstein distance of $X$ and $\mathbf{e}$. 

\begin{theorem}\label{theo:pekoz}
	Let $X$ be a non-negative random variable such that ~$\mathbb{E}[X]=m$ and ~$\mathbb{E}[X^2]<\infty$. Let $X^e$  be the  equilibrium distribution of $X$ and $\mathbf{e}$ be a standard exponential random variable. Then, 
	\begin{equation*}
		d_W\left(X,\mathbf{e}\right) \leq 2\mathbb{E}\big[|X-X^e|\big]+ |m-1|.
	\end{equation*}
\end{theorem}

\begin{proof}
	Let $f\in \mathcal{F}_W$. Since $f(0)=0$ we have $m \mathbb{E}[f^{\prime}(X^e)]=\mathbb{E}[f(X)]$. Therefore,
	\begin{equation*}
		\begin{split}
			\left|\mathbb{E}[f'(X)-f(X)]\right|&\leq \left|\mathbb{E}[f'(X)-f'(X^e)]\right|+|1-m|\ |\mathbb{E}[f'(X^e)]|\\
			&\leq \|f''\|\ \mathbb{E}[|X-X^e|]+|1-m| \ \|f'\|.
		\end{split}
	\end{equation*}
	The result holds  by taking supremum  over $f\in \mathcal{F}_W$ and Theorem \ref{theo:5.4}. 
\end{proof}	

The previous theorem was proved by  Peköz and Röllin under the assumptions  
$\mathbb{E}[X]=1$ and $\mathbb{E}[X^2]<\infty$, as shown in \cite[Theorem 2.1]{pekoz2011}, and another similar proof can be found in \cite[Lemma 5.6]{MR2861132}. From the previous result and Lemma \ref{lem:sizebiastoequilibrium}, we obtain the following bound in terms of the size-biased distribution:
\begin{equation}\label{eq: bound sizebiased}
	d_W\left(X,\mathbf{e}\right) \leq 2\mathbb{E}\big[|X-U\dot{X}|\big]+ |m-1|.
\end{equation}

The above bound shifts the focus of the problem towards the description of the size-biased distribution of $Z_n$  conditioned on $\{Z_n>0\}$. In this regard, let 
$\tilde{X}$ denote the law of $X$ conditioned on the event $\{ X>0\}$. Using the fact that every measurable and bounded function  $f:\R\rightarrow\R$ satisfies 
\begin{align*}
	\E[f(\dot{X})]
	&=\frac{\E[Xf(X)]}{\E[X]}=\frac{\E[Xf(X)\mid X>0]}{\E[X\mid X>0]}=\frac{\E[\tilde{X}f(\tilde{X})]}{\E[\tilde{X}]},
\end{align*}
we deduce the following simple but useful result.
\begin{lemma}\label{lem:conditionedbias}
	Let $X$ be a non-negative random variable with finite mean and let $Y$ be $X$ conditioned on $\{ X>0\}$. Then the size-biased distributions of $X$ and $Y$ coincide.
\end{lemma}

\subsection{Rates on Yaglom's limit in a constant environment}
\label{proof constant}

In this section, we show Theorem \ref{thm:PekozRoellin}. The proof basically follows  the same arguments as those used in the proof of  \cite[Theorem 3.3]{pekoz2011}, with a slight modification. We use their notation and encourage  the reader to read our proof alongside theirs for a better understanding. 

\begin{proof}[Proof of Theorem  \ref{thm:PekozRoellin}]
	Let $\{Z_n:n\geq 0\}$ be  a critical Galton-Watson process with constant environment such that $\sigma^2<\infty$ and $\E[(Z_1)^3]<\infty$. In what follows, $C_1, C_2, \dots$ are strictly positive constants independents of $n$. We begin by showing the first bound. We define the random variables $R_n^*, R_n, U$ as in \cite[Theorem 3.3]{pekoz2011}, as well as  
	\[W := \frac{2}{\sigma^2 n} R_n^* \qquad \text{and}\qquad W^e := \frac{2}{\sigma^2 n} (R_n - U).\]
	From \cite[Theorem 3.3]{pekoz2011}, we notice  that $W$ and $\{2(\sigma^2n)^{-1}Z_n\mid Z_n>0\}$ have the same law and 
	\[\mathbb{E}\big[|W- W^e|\big] \le C_1 \frac{\log (n)}{n}.\]
	Further, from \eqref{eq:kolmesti} we obtain that 
	\begin{equation*}
		\left|\E\left[\left.\frac{2}{\sigma^2n}Z_n\ \right| Z_{n}>0\right] - 1\right| \le C_2 \frac{{\log(n)}}{n}.
	\end{equation*}
	With this in hand, the claim follows from Theorem \ref{theo:pekoz}.  Indeed, 
	\begin{equation*}
		d_W\left(\left\{\left.\frac{2}{\sigma^2n}Z_n\ \right| Z_{n}>0\right\}, \mathbf{e}\right) \leq 2\mathbb{E}\big[|W-W^e|\big]+	\left|\E\left[\left.\frac{2}{\sigma^2n}Z_n\ \right| Z_{n}>0\right] - 1\right| \leq C_3 \frac{{\log(n)}}{n}.
	\end{equation*}
	
	For the second bound, we make the following modification in  the proof of \cite[Theorem 3.3]{pekoz2011}. The random variables $W$  and $W^e$  are now defined as follows
	\[W:= \mathbb{P}[Z_n>0] R_n^* \qquad \text{and} \qquad  W^e := \mathbb{P}[Z_n>0](R_n-U).\]
	Then, the claims (i)-(vii) about $R_n$ and $R_n^*$ in the proof of  \cite[Theorem 3.3]{pekoz2011}, imply that 
	\[\frac{1}{\mathbb{P}[Z_n>0]} \mathbb{E}\big[|W- W^e|\big] \le C_5  \log (n).\]
	Now, appealing to Kolmogorov's estimate \eqref{eq: kolmogorov estimate}, 
	we have that  $\lim_{n\to \infty} n\mathbb{P}[Z_n>0] = 2/\sigma^2$, so
	\[\mathbb{E}\big[|W- W^e|\big] \le C_6  \frac{\log (n)}{n}.\]
	Since $\mathbb{E}[ \{\mathbb{P}[Z_n>0]~ Z_n \mid Z_n>0\} ]=1$, we can apply \cite[Theorem 2.1]{pekoz2011} (or equivalently Theorem \ref{theo:pekoz}) to conclude the desired bound. 
\end{proof}

\subsection{The size-biased distribution of $Z_{n}$}\label{sec:sizebiasdist}
In this subsection, we examine the size-biased process $\dot{Z}:=\{\dot{Z}_n:n\geq 0\}$
and establish a connection between $Z_n$ 
and the population size at generation $n$ of a  random tree. The material presented here is comprehensively explained in \cite[Sections 1.4.1 and 1.4.2]{kersting2017discrete} and \cite[Section 2 and 3]{cardona2021yaglom}.  
We will only present the key ideas to enhance the readability of our manuscript and ensure its self-contained nature.

We  recall that there exists a relationship between Galton-Watson processes  and Galton-Watson trees, both of which have a varying environment 
$Q=\{q_n:  n\geq 1\}$.  Specifically, 
$Z_n$ represents the number of particles at generation $n$ of the random tree constructed as 
follows: any particle in generation $i$ gives birth to particles in generation $i + 1$ with distribution $q_{i+1}$.

For this purpose, we introduce a labelling of particles, known as the {\it Ulam-Harris labelling}, which directly reveals ancestral relationships. In this approach, particles are identified by elements $u$ in $\mathcal{U}$,   the set of finite sequences of strictly positive integers, including  $\emptyset$.  For  $ u\in \mathcal{U}$, we define the length of  $u$ as   $|u|:=n$, if  $u=u_1\cdots u_n$, where $n\geq 1$  and  $|\emptyset|:=0$ if $u=\emptyset$.   If  $u$  and  $v$  are two elements in  $\mathcal{U}$, we denote their concatenation as  $uv$, with the convention  that  $uv=u$  if  $v=\emptyset$.  The genealogical line of  $u$  is denoted by  $[\emptyset,u]=\{\emptyset\}\cup\{u_1\cdots u_j: j=1,\ldots ,n\}$.

A {\it rooted tree}  $\textbf{t}$  is a subset of  $\mathcal{U}$  that satisfies   $\emptyset\in\textbf{t}$,  $[\emptyset, u]\subset\textbf{t}$ for any $u\in\textbf{t}$,  and if  $u\in \textbf{t}$  and  $i\in\mathbb{N}$  satisfy  $ui\in\textbf{t}$,  then  $uj\in\textbf{t}$  for all  $1
\leq j\leq i$. We  denote by  $\mathcal{T} = \{\textbf{t}: \textbf{t}\ \mbox{is a tree}\}$ the subspace of rooted trees.  The vertex  $\emptyset$  is called {\it the root} of the tree. For any  $u\in \textbf{t}$,  we define the number of children of $u$ as $l_u(\textbf{t})=\max\{i\in\mathbb{Z}^+: ui\in\textbf{t}\}$.  The height of  $\textbf{t}$  is defined by $|\textbf{t}|= \sup\{|u|: u \in \textbf{t}\}$.  For any  $n\in \mathbb{N}$  and trees $\textbf{t}$ and $ \tilde{\textbf{t}}$, we write  $\textbf{t}\overset{n}{=}\tilde{\textbf{t}}$  if they coincide up to height $n$. The population size in the  $n$-th generation of the tree  $\textbf{t}$  is denoted by  $X_n(\textbf{t})=\#\{u\in\textbf{t}: |u|=n\}$.

A {\it Galton-Watson tree in a environment} $Q=\{q_n:  n\geq 1\}$  is a  $\mathcal{T}$-valued random variable  $\textbf{T}$  such that 
\begin{equation}\label{eq: Galton watson tree}
	\mathbb{P}[\textbf{T}\overset{n}{=}\textbf{t}] = \prod_{u\in \textbf{t}:\, |u|<n}  q_{|u|+1}[l_u(\textbf{t})],
\end{equation}
for any  $n\geq 0$  and any tree  $\textbf{t}$. As mentioned before, the process  $Z=\{Z_n:n\geq 0\}$  defined as  $Z_n=X_n(\textbf{T})$  is a Galton-Watson process in a environment $Q$.

In a similar way,  $\dot{Z}_n$\ is the population size at generation $n$ of some random tree.  
According to \cite[Sections 1.4.1 and 1.4.2]{kersting2017discrete}, the tree associated with $\dot{Z}$ is a size-biased tree in a varying environment $Q$. The {\it size-biased tree in a varying environment} $Q$\ is constructed similarly to a Galton-Watson tree, but
it incorporates additional information regarding a \textit{marked particle} that is distinguished from the rest. This marked particle is ultimately forced to fit the definition of $\dot{Z}_{n}$. More precisely,  let  $\dot{q}_i$ be the size-biased distribution of $q_i$, for each $i\geq 1$. We define  a size-biased tree in a varying environment $Q$ as follows:
\begin{enumerate}
	\item[(i)] We start with an initial marked particle.
	\item[(ii)] The marked particle in generation $i \geq 0$, gives birth to particles in generation $i+1$ according to $\dot{q}_{i+1}$.  Uniformly over the progeny, we select one of the individuals and classify it as a marked particle. The remaining members of the offspring remain unmarked.
	\item[(iii)] Any unmarked particle in generation $i \geq 0$, gives birth to unmarked particles in generation $i+1$ according to $q_{i+1}$, independently of other particles.
\end{enumerate}

The marked genealogical line is known as the {\it spine}. This construction is recognized as the {\it one-spine decomposition} of the tree. Now, we associate a probability measure in the set of rooted trees to the size-biased tree.  First, we require a probability measure on the set of rooted trees with one spine. 
A {\it spine} or distinguished path $\textbf{v}$ on a tree $\textbf{t}$ is a sequence  $\{v^{(k)}: k=0,1,\ldots,|\textbf{t}|\}\subset \textbf{t}$ (or $\{v^{(k)}: k=0,1,\ldots\}\subset \textbf{t}$ if $|\textbf{t}|=\infty$) such that $v^{(0)}=\emptyset$ and  $v^{(k)}=v^{(k-1)}j$ for some $j\in\mathbb{N}$, for any $1\leq k\leq |\textbf{t}|$.  We denote by $\dot{\mathcal{T}}$, the {\it subspace of trees with one spine} given by 
$$\dot{\mathcal{T}}= \{(\textbf{t},\textbf{v}):\ \textbf{t}\ \mbox{is a tree and } \textbf{v}\ \mbox{is a spine on}\ \textbf{t} \},$$ 
and by $\mathcal{T}_n=\{\textbf{t}\in \mathcal{T}: |\textbf{t}|=n\}$ and $\dot{\mathcal{T}}_n=\{(\textbf{t},\textbf{v})\in \dot{\mathcal{T}}: |\textbf{t}|=n\}$ the restriction of $\mathcal{T}$ and $\dot{\mathcal{T}}$ to trees with height $n$.  Then, the size-biased tree with one spine can be seen as a $\dot{\mathcal{T}}$-valued random variable $ (\dot{\textbf{T}},\textbf{V})$ with distribution
\begin{equation}\label{eq:lawtv}
	\begin{split}
		\mathbb{P}[(\dot{\textbf{T}},\textbf{V})\overset{n}{=}(\textbf{t},\textbf{v})] :&= \prod_{u\in \textbf{v}:\, |u|<n}  \dot{q}_{|u|+1}[l_u(\textbf{t})]\frac{1}{l_u(\textbf{t})}\prod_{u\in \textbf{t}\setminus\textbf{v}:\, |u|<n}  q_{|u|+1}[l_u(\textbf{t})]\\
		&= \prod_{u\in \textbf{v}:\, |u|<n}  \frac{q_{|u|+1}[l_u(\textbf{t})]}{f'_{|u|+1}(1)}\prod_{u\in \textbf{t}\setminus\textbf{v}:\, |u|<n}  q_{|u|+1}[l_u(\textbf{t})]\\
		&=\ \frac{1}{\mu_n}\mathbb{P}[\textbf{T}\overset{n}{=}\textbf{t}] 
	\end{split}
\end{equation}

for any $n\geq 0$ and  $(\textbf{t},\textbf{v})\in\dot{\mathcal{T}}_n$, where we have used the definition of $\dot{q}$, \eqref{eq:mu} and  \eqref{eq: Galton watson tree}. 
Hence, by summing over all the possible spines, we obtain that  the distribution of the {\it size-biased Galton-Watson  tree  in an environment} $Q$ on $\mathcal{T}$ is given by 
\begin{equation}\label{eq:lawt}
	\mathbb{P}[\dot{\textbf{T}}\overset{n}{=}\textbf{t}]=  \sum_{\textbf{v}: (\textbf{t},\textbf{v})\in\dot{\mathcal{T}}_n} \mathbb{P}[(\dot{\textbf{T}},\textbf{V})\overset{n}{=}(\textbf{t},\textbf{v})] =  \frac{1}{\mu_n} X_n(\textbf{t})\mathbb{P}[\textbf{T}\overset{n}{=}\textbf{t}] ,
\end{equation}
for any $n\geq 0$ and any $\textbf{t}\in\mathcal{T}_n$  (see also \cite[Lemma 1.2]{kersting2017discrete}). Let    $\dot{Z}=\{\dot{Z}_n: n\geq 0 \}$ be the process defined as  $\dot{Z}_n=X_n(\dot{\textbf{T}})$, for each $n\geq 0$. 
By the previous equation, we can see that $\dot{Z}_n$ is  the size-biased distribution of $Z_n$ for each $n\geq 0$.

Lemma \ref{lem:conditionedbias} tells us that $\dot{Z}_n$ is also the size-biased distribution of $Z_n$ conditioned on $\{Z_n>0\}$, which implies that $\dot{Z}_n>0$ for all $n\geq 0$ and 
$|\dot{\textbf{T}}|=\infty$. Let $\textbf{V}=\{V^{(k)}: k=0,1,\ldots\}$ be the associated spine.  For each $n\geq 0$, denote by $L_n$ and $R_n$ the number of particles at the left (excluding $V^{(n)}$) and at the right (including $V^{(n)}$), respectively, of particle $V^{(n)}$ in $\dot{\textbf{T}}$. Then, $$\dot{Z}_n=L_n+R_n, \qquad n\geq 0.$$
We can classify these particles according to the generation at which they split off from the spine. More precisely,  for each ~$j \in \{1,\dots,n\}$, denote by ~$\dot{Z}_{n,j}$ the number of particles at generation $n$  in $\dot{\textbf{T}}$ that are descendants of the siblings of  $V^{(j)}$ but not $V^{(j)}$ itself. In the same spirit, denote by ~$L_{n,j}$~  and ~$R_{n,j}$~ the number of particles at generation $n$ in $\dot{\textbf{T}}$ that stem from the siblings to the left and right, respectively, of  $V^{(j)}$ (both excluding $V^{(j)}$).   
Note that for fixed ~$j  \in \{1,\dots,n\}$, the random variables ~$L_{n,j}$ and ~$R_{n,j}$ are in general not independent, as they are linked through the offspring number of $V^{(j-1)}$. However, by the branching property, the couples $(L_{n,1}, R_{n,1}), \dots ,(L_{n,n}, R_{n,n})$ are independent. For each $n\geq 1$, we have the following 
\begin{equation}\label{eq: Lnj, Rnj}
	L_n= \sum_{j=1}^{n} L_{n,j}, \qquad  R_n = 1+ \sum_{j=1}^{n} R_{n,j}\qquad \text{and}\qquad  	\dot{Z}_{n}=1+\sum_{j=1}^{n}(L_{n,j}+ R_{n,j}).
\end{equation}

\subsection{The shape function}\label{sec: shape}
Recall that  $f_n$ denotes the generating function associated with $q_n$. 
For each $0\leq m< n$ and $s\in [0,1]$, we define 
$ f_{m,n}(s):= [f_{m+1} \circ \cdots \circ f_n](s)$ and  $f_{n,n}(s):=s$. According to  equation \eqref{laplace Z}, the generating function of  $Z_n$  is equal to  $f_{0,n}$. 
An essential part of our proof is to understand 
$\{f_{0,n}: n\geq 1\}$. To handle such iterated compositions of generating functions, we utilize the shape function, a common device in the theory of branching processes. For a generating function $f$, 
we define the shape function $\varphi:[0,1)\rightarrow \mathbb{R}$ associated with $f$ as   
\begin{equation*}
	\varphi(s): = \frac{1}{1-f(s)}-\frac{1}{f'(1)(1-s)} , \quad 0\leq s < 1.
\end{equation*} 
Due to the convexity of $f$,  the function $\varphi$ is non-negative. Moreover,  by Taylor expansion of $f$ around 1, one can extend the definition of $\varphi$ to 1 as 
\begin{equation*}
	\varphi(1) :=\lim\limits_{s\uparrow 1} \varphi(s)= \frac{f''(1)}{2f'(1)^2}.
\end{equation*}
 \cite{kersting2020} showed in Lemma 1 and 6 that  under condition \eqref{cond:Kersting}, there exists a $C>0$ such that  
\begin{equation}\label{eq:bound varphi}
	C\varphi_l(1)\leq \varphi_l(s)\leq \varphi_l(1), \qquad \mbox{ for every } s\in[0,1],\ l\geq 1,
\end{equation}
where $\varphi_l$ is the shape function of $f_l$. 
According to \cite[Lemma 5]{kersting2020}, for every $0\leq k<n$, the shape function $\varphi_{k,n}$ of $f_{k,n}$ is  given by
$$\varphi_{k,n}(s):=\mu_k\underset{l=k}{\overset{n-1}{\sum}}\ \frac{\varphi_l(f_{l,n}(s))}{\mu_{l-1}}, \qquad \qquad s\in[0,1).$$
For further details about the shape function see for instance \cite[Section 3]{kersting2020}.

\section{Proofs of main theorems}\label{sec:proofs}
In this section, we prove our main results. Let $Q$ be {a} varying environment such that conditions \eqref{cond:thirdmoment} and \eqref{eq:critical} are satisfied. In the sequel, $C$ will denote a generic constant independent of $n$, whose value might vary from one line to another. Denote by $Y_n$ the law of $Z_n$ conditioned on $\{Z_n>0\}$ and let $b_n:=\E[Y_n]=\E[Z_n\mid Z_n>0]$ its mean. By equation  \eqref{eq: bound sizebiased},
we have 	
\begin{equation}\label{eq: bound Y_n}
	d_W\left(\frac{Y_n}{b_n},\mathbf{e}\right) \leq 2\mathbb{E}\left[\left|\frac{Y_n}{b_n}-U\frac{\dot{Y}_n}{b_n}\right|\right]=\frac{2}{b_n}\mathbb{E}\left[\left|Y_n-U\dot{Y}_n\right|\right],
\end{equation}
where $\dot{Y}_n$ is the size biased distribution of $Y_n$.
According to Lemma \ref{lem:conditionedbias}, $\dot{Y}_n$ coincides with  $\dot{Z}_n$, the size-biased distribution of $Z_{n}$. Therefore,  in order to  have a bound, it is natural to embed $Y_n$ and $U\dot{Y}_n$ into the size-biased tree $\dot{\textbf{T}}$ in a varying environment $Q$.

We start with the embedding  of $U\dot{Y}_n$. Let $\textbf{V}$ {be} the associated spine. According to \eqref{eq:lawtv} and \eqref{eq:lawt}, $V^{(n)}$ is {selected uniformly} between the individuals at generation $n$. This implies that    $R_n$, the number of particles at the right (including) of particle $V^{(n)}$ in $\dot{\textbf{T}}$, is uniform in $\{1,\dots, \dot{Z}_n\}$.  This discrete uniformity can be adjusted by an  additive continuous uniform variable.  Let $U$ be a Uniform$[0,1]$  random variable independent of the Galton-Watson tree. We have the following identity 
\begin{equation}\label{eq:equiliY}
	U \dot{Y}_n = U \dot{Z}_n\overset{(d)}{=} R_n - U=  1-U+ \sum_{j=1}^{n} R_{n,j},
\end{equation}
where we have used \eqref{eq: Lnj, Rnj}.

Now, we will construct {an embedding} of $Y_n$ into $\dot{\textbf{T}}$. Recall $\{Z_n\mid Z_n>0\}$ is the  population size at generation $n$ of a Galton-Watson tree with alive individuals in that generation, while $\dot{Z}_n$ is the population size at generation $n$ of a tree where we select one spine uniformly at random from the particles alive in that generation.  Then, a way to undo this size-biased and obtain $\{Z_n\mid Z_n>0\}$ in term{s} of $\dot{Z}_n$ is to condition that the spine is a \textit{specific} {live} particle, for example the leftmost particle, i.e. 
\begin{equation}\label{eq:copyYn}
	Y_n \overset{(d)}{=} \{\dot{Z}_n \mid  L_n=0\}.
\end{equation}
Indeed,  we use Bayes' formula, the definition of size-biased distribution, and that $V^{(n)}$ is uniform in $\{1,\dots,\dot{Z}_n\}$,  to obtain 
\begin{equation*}
	\begin{split}
		\mathbb{P}\big[\dot{Z}_n = k \mid L_n= 0\big]
		&=	 \frac{	\mathbb{P}\big[L_n=0 \mid \dot{Z}_n = k\big]\mathbb{P}\big[ \dot{Z}_n =k\big]}{\sum_{j=1}^{\infty}\mathbb{P}[L_n=0 \mid  \dot{Z}_n =j] \mathbb{P}[\dot{Z}_n = j]}\\ 
		& =\frac{	k^{-1} \ k\mathbb{P}\big[ Z_n =k\big] (\mu_n)^{-1}}{\sum_{j=1}^{\infty}	j^{-1}  j \mathbb{P}\big[ Z_n =k\big] (\mu_n)^{-1}} \\
		&= \frac{\mathbb{P}\big[ Z_n =k \big]}{ \mathbb{P}[Z_n>0]}= \mathbb{P}\left[Y_n=k\right],
	\end{split}
\end{equation*}
for every $k\geq 1$. Recall the definitions of $R_n, L_n,  R_{n,j}$ and $ L_{n,j}$ from Section \ref{sec:sizebiasdist}. For each $j\leq n$ define the event $A_{n,j}:=\{L_{n,j}=0\}$.  By the branching property, the events $A_{n,1}, \dots, A_{n,n}$ are independent and $$\{L_n=0\}=\underset{r=1}{\overset{n}{\bigcap}}\ A_{n,r}, \qquad \qquad n\geq 1.$$
Moreover, by \eqref{eq: Lnj, Rnj} and the independence of the couples $\left\{(L_{n,j}, R_{n,j}): 1\leq j\leq n\right\}$,  we have
\begin{equation}\label{eq: Z in terms Rn,j}
	\{\dot{Z}_n \mid  L_n=0\}=\left\{1+\sum_{j=1}^{n} R_{n,j} ~ \Big| ~ L_n=0\right\}\overset{(d)}{=}1+\sum_{j=1}^{n} \left\{R_{n,j} \mid  A_{n,j}\right\}.
\end{equation}
We want to give a better formulation of  $\left\{R_{n,j} \mid  A_{n,j}\right\}$.  Denote by $A_{n,j}^c$ the complement of $A_{n,j}$. Observe that for every $k\geq 1$
\begin{equation*}
	\mathbb{P}\left[R_{n,j}=k\mid A_{n,j}\right]=\frac{ \mathbb{P}\left[R_{n,j}=k, A_{n,j}\right]}{ \mathbb{P}\left[ A_{n,j}\right]}=\mathbb{P}\left[R_{n,j}=k, A_{n,j}\right]+\frac{\mathbb{P}\left[R_{n,j}=k, A_{n,j}\right] }{ \mathbb{P}\left[ A_{n,j}\right]}\mathbb{P}\left[ A_{n,j}^c\right].
\end{equation*}
Consider a sequence $\{\widetilde{R}_{n,j}: 1\geq j\geq n\}$ of random variables independent of the size-biased tree $\dot{\textbf{T}}$ such that $\widetilde{R}_{n,j}:\overset{(d)}{=} \left\{R_{n,j} \mid  A_{n,j}\right\}$ for each $1\leq j\leq n$. Then, from the previous computation, we have 
$$\left\{R_{n,j} \mid  A_{n,j}\right\}\overset{(d)}{=}R_{n,j} \Indi{A_{n,j}}+\widetilde{R}_{n,j} \Indi{A_{n,j}^c}, \qquad 1\leq j\leq n.$$

The above analysis, together with \eqref{eq: bound Y_n},  \eqref{eq:equiliY}, \eqref{eq:copyYn}  and \eqref{eq: Z in terms Rn,j} imply 
\begin{equation}\label{eq:meanYYe}
	d_W\left(\frac{Y_n}{b_n},\mathbf{e}\right) 
	\leq \frac{2}{b_n}\left(\mathbb{E}\left[U\right]
	+ \sum_{j=1}^{n} \mathbb{E}\left[\widetilde{R}_{n,j} \mathds{1}_{A_{n,j}^c} + R_{n,j}\mathds{1}_{A_{n,j}^c} \right]\right).
\end{equation}
In the next subsections, we are going to analyse  the previous bound in order to give the precise rate of decay in Theorems \ref{theo:speedYaglom} and  \ref{theo:rnestimation}.  

\subsection{Proof of Theorem \ref{theo:speedYaglom}}\label{sec:proofs1}
First, let us introduce some notation.  Given an environment $Q=\{q_n: n\ge 1\}$ and $m\geq 0$, we define the shifted environment as $Q_m:=\{q_{m+n}: n\ge 1\}$. 
Let $Z=\{Z_n: n\geq 0\}$ be  a Galton-Watson process in {a} environment $Q$. The branching property of $Z$, implies that for every $m\geq 0$, 
\begin{equation}\label{eq:branchingproperty}
	\{Z_{m+n}, n\geq 0 \mid Z_m=1\} \mbox{ is a Galton-Watson process in an environment } Q_m.
\end{equation}
In some parts,  we need to \textit{keep track of the environment}. In these occasions, we are going to denote a Galton-Watson process in an environment $Q$ by $\{Z_n^Q: n\geq 0\}$.  Observe  that 
\begin{equation*}
	\mathbb{E}\left[s^{Z_n^{Q_m}}\right] = f_{m,n}(s), \qquad 0 \leq s \leq 1.
\end{equation*}
By taking derivatives with respect to the variable $s$,  evaluating the resulting function at $s=1$ and using \cite[equation (17)]{cardona2021yaglom} with $s=1$ we obtain
\begin{equation}\label{eq:media_Zn}
	\mathbb{E}\big[Z_{n}^{Q_m}\big]=\frac{\mu_{m+n}}{\mu_m} \qquad \mbox{and} \qquad \frac{\mathbb{E}\big[Z_n^{Q_m}(Z_n^{Q_m}-1)\big]}{\mathbb{E}\big[Z_n^{Q_m}\big]^2}=\mu_{m}\left(\sum_{k=m}^{m+n-1}\frac{\nu_{k+1}}{\mu_{k}} \right).
\end{equation}

We recall that in the construction of the size biased tree $\dot{\textbf{T}}$, the marked particle in generation $i$, gives birth to particles in generation $i+1$ according to $\dot{q}_{i+1}$.  Uniformly over the progeny, we select one of the individuals and classify it as a marked particle. The remaining members of the offspring remain unmarked. Therefore, the marked particle gives birth to $k$  unmarked particles with probability $\dot{q}_{i+1}[k+1]$. This new distribution will be denoted by $[\dot{q}_{i+1}-1]$, i.e. for all $k\geq 0$,  $[\dot{q}_{i+1}-1][k]=\dot{q}_{i+1}[k+1]$.

The proof is divided into five steps. In the first one and second one, we find bounds for $\mathbb{E}\left[\widetilde{R}_{n,j} \mathds{1}_{A_{n,j}^c} \right]$ and $\mathbb{E}\left[ R_{n,j} \mathds{1}_{A_{n,j}^c} \right]$ in  terms of $\mathbb{P}[A_{n,j}^c]$ and $\mathbb{P}\big[Z_{n-j}^{Q_j}>0\big]$. In the third and four steps, we find a bound  for these probabilities. In the last one we combine all our findings for obtaining the desired result. \\

\noindent \textbf{Step I:}
Our first goal is to show that for each ~$1\leq j\leq n$, we have 
\begin{equation}\label{lem:lemRnjprime}
	\mathbb{E}\Big[\widetilde{R}_{n,j}\mathds{1}_{A^c_{n,j}}\Big] \leq \mu_n \frac{\nu_j}{\mu_{j-1}}\mathbb{P}[A_{n,j}^c].
\end{equation}
By definition, $\widetilde{R}_{n,j}$ and $A_{n,j}$ are independent and $\widetilde{R}_{n,j}\overset{(d)}{=} \left\{R_{n,j} \mid  A_{n,j}\right\}=\left\{\dot{Z}_{n,j} \mid  A_{n,j}\right\}$. Then,
\begin{equation*}
	\begin{split}
		\mathbb{E}\Big[\widetilde{R}_{n,j}\mathds{1}_{A^c_{n,j}}\Big]
		=\frac{\mathbb{E}\big[\dot{Z}_{n,j}\mathds{1}_{A_{n,j}}\big]}{\mathbb{P}[A_{n,j}]}\mathbb{P}[A_{n,j}^c].
	\end{split}	
\end{equation*}
Since $\dot{Z}_{n,j}=R_{n,j}+L_{n,j}$ and $A_{n,j}=\{L_{n,j}=0\}$, the variables ~$\mathds{1}_{A_{n,j}}$ and ~$\dot{Z}_{n,j}$ are negatively correlated, which yields
$$ \mathbb{E}\big[\dot{Z}_{n,j}\mathds{1}_{A_{n,j}}\big] \leq  \mathbb{E}\big[\dot{Z}_{n,j}\big]\mathbb{P}[A_{n,j}].$$
It follows that
\begin{equation*}
	\mathbb{E}\Big[\widetilde{R}_{n,j}\mathds{1}_{A^c_{n,j}}\Big] = \frac{\mathbb{E}\big[\dot{Z}_{n,j}\mathds{1}_{A_{n,j}}\big]}{\mathbb{P}[A_{n,j}]}\mathbb{P}[A_{n,j}^c] \leq  \mathbb{E}\big[\dot{Z}_{n,j}\big]\mathbb{P}(A_{n,j}^c).
\end{equation*}
It remains to  find the upper estimate for the expectation of ~$\dot{Z}_{n,j}$. Now, we need to keep track on the environment and we are going to denote $\{Z_n^Q: n\geq 0\}$ a Galton-Watson process in an environment $Q$. Since  $\dot{Z}_{n,j}$ the number of particles at generation $n$  in $\dot{\textbf{T}}$ that are descendants of the siblings of  $V^{(j)}$ but not $V^{(j)}$ itself, we have that 
\begin{equation*}
	\begin{split}
		\mathbb{E}\big[\dot{Z}_{n,j}^Q\big] &= \mathbb{E}\Big[\dot{Z}_n ^Q\ \Big|\Big. \  \dot{Z}_{j-1}^Q=1\Big] -  \mathbb{E}\Big[\dot{Z}_n ^Q\ \Big|\Big. \  \dot{Z}_{j}^Q=1\Big]\\
		&= \mathbb{E}\Big[\dot{Z}_{n-(j-1)}^{Q_{j-1}}-1\Big] - \mathbb{E}\Big[\dot{Z}_{n-j}^{Q_j}-1\Big] \\ &=\frac{\mathbb{E}\Big[Z_{n-(j-1)}^{Q_{j-1}}\big(Z_{n-(j-1)}^{Q_{j-1}}-1\big)\Big]}{\mathbb{E}\Big[Z_{n-(j-1)}^{Q_{j-1}}\Big]} - \frac{\mathbb{E}\Big[Z_{n-j}^{Q_{j}}\big(Z_{n-j}^{Q_{j}}-1\big)\Big]}{\mathbb{E}\Big[Z_{n-j}^{Q_{j}}\Big]},
	\end{split}
\end{equation*}
where we have used the branching property \eqref{eq:branchingproperty} and the size-biased distribution in the last two lines.  We use  \eqref{eq:media_Zn} for the environments $Q_{j-1}$ and $Q_{j}$ to conclude that
\begin{equation*}
	\begin{split}
		\mathbb{E}\big[\dot{Z}_{n,j}^Q\big]  &=  \mu_n \sum_{k=j-1}^{n-1}\frac{\nu_{k+1}}{\mu_{k}} -\mu_n \sum_{k=j}^{n-1}\frac{\nu_{k+1}}{\mu_k} = \mu_n \frac{\nu_j}{\mu_{j-1}}.
	\end{split}
\end{equation*}
This completes the proof of \eqref{lem:lemRnjprime}.\\

\noindent \textbf{Step II:}
Our second goal is to show that for each ~$1\leq j\leq n$, we have 
\begin{equation}\label{eq:StepII}
	\mathbb{E}\Big[R_{n,j}\mathds{1}_{A_{n,j}^c}\Big] \leq \frac{\mu_n}{\mu_j}\left(\frac{f_j'''(1)+ f_j''(1)}{f_j'(1)}\right)\mathbb{P}\big[Z_{n-j}^{Q_j}>0\big].
\end{equation}

Let $\textbf{V}=\{V^{(k)}: k=0,1,\ldots\}$ be the spine associated {with} $\dot{\textbf{T}}$.
In the sequel, for each ~$j\geq 1$, we denote  $S_j$  the random number of siblings of  $V^{(j)}$. By  construction of $\dot{\textbf{T}}$, we know that $S_j$ is distributed as $[\dot{q}_j-1]$. Moreover, there exists $I_j$ uniformly distributed on $\{1, \dots, S_j+1\}$ such that $V^{(j)}=V^{(j-1)}I_j$. 
We observe that, conditional on $S_j$ and $I_j$, the random variables $R_{n,j}$ and $L_{n,j}$ are independent. Hence,

\begin{equation*}
	\mathbb{E}\Big[R_{n,j}\mathds{1}_{A_{n,j}^c}\Big] = \mathbb{E}\Big[\mathbb{E}\big[R_{n,j}\mathds{1}_{A_{n,j}^c} \big|\big. S_j, I_j\big]\Big] = \mathbb{E}\Big[ \mathbb{E}\big[R_{n,j}\mid S_j, I_j\big] \mathbb{E}\big[\mathds{1}_{A_{n,j}^c} \mid S_j, I_j\big]\Big].
\end{equation*}
Further, by construction of $\dot{\textbf{T}}$ and the branching property \eqref{eq:branchingproperty}, we obtain
\begin{equation*}
	\mathbb{E}\Big[R_{n,j}\mid S_j, I_j\Big] \leq S_j \mathbb{E}\big[Z_{n}^Q\mid Z_j^Q=1 \big]=S_j \mathbb{E}\big[Z_{n-j}^{Q_j}\big],
\end{equation*}
and 
\begin{equation}\label{eq: A_njc condioned S}
	\mathbb{E}\Big[\mathds{1}_{A_{n,j}^c}\mid S_j, I_j\Big] =
	\mathbb{P}\big[L_{n,j}\not= 0 \mid  S_j, I_j\big]\leq
	S_j\mathbb{P}\big[Z_{n}^Q>0\mid Z_j^Q=1 \big]= S_j \mathbb{P}\big[Z_{n-j}^{Q_j}>0\big].
\end{equation}
By putting both terms in the previous expectation, we have 
\begin{equation*}\label{eq:RA}
	\mathbb{E}\Big[R_{n,j}\mathds{1}_{A_{n,j}^c}\Big]  \leq \mathbb{E}\Big[ S_j^2 \Big]\mathbb{E}\big[Z_{n-j}^{Q_j}\big] \mathbb{P}\big[Z_{n-j}^{Q_j}>0\big].
\end{equation*}
Now,  we use  that $S_j$ is distributed as $[\dot{q}_j-1]$ and the size-biased distribution of $Z^{Q_j}_1$, to show that 
\begin{equation*}
	\begin{split}
		\mathbb{E}\big[S_j^2\big] &=  \mathbb{E}\big[(\dot{Z}_1^{Q_{j-1}}-1)^2\big] =\frac{ \mathbb{P}\big[(Z_1^{Q_{j-1}}-1)^2 Z_1^{Q_{j-1}}\big] }{\mathbb{P}\big[Z_1^{Q_{j-1}}\big]}\\ &=  \frac{ \mathbb{P}\big[(Z_1^{Q_{j-1}}-2)(Z_1^{Q_{j-1}}-1) Z_1^{Q_{j-1}}\big] + \mathbb{P}\big[(Z_1^{Q_{j-1}}-1) Z_1^{Q_{j-1}}\big]}{\mathbb{P}\big[Z_1^{Q_{j-1}}\big]}= \frac{f_j'''(1)+ f_j''(1)}{f_j'(1)}.
	\end{split}
\end{equation*}
Therefore, by using \eqref{eq:media_Zn} with the environment $Q_j$, the claim holds
\begin{equation*}
	\begin{split}
		\mathbb{E}\Big[R_{n,j} \mathds{1}_{A_{n,j}^c}\Big] 
		& \leq \mathbb{E}\Big[ S_j^2 \Big]\mathbb{E}\big[Z_{n-j}^{Q_j}\big] 
		\mathbb{P}\big[Z_{n-j}^{Q_j}>0\big]
		= \left(\frac{f_j'''(1)+ f_j''(1)}{f_j'(1)}\right)\frac{\mu_n}{\mu_j}\mathbb{P}\big[Z_{n-j}^{Q_j}>0\big].
	\end{split}
\end{equation*}\\

\noindent \textbf{Step III:} We will show that for each  ~$1\leq  j \leq  n$,
\begin{align}\label{lem:Anj}
	\mathbb{P}[A_{n,j}^c]
	\leq \frac{f_j''(1)}{f_j'(1)}\mathbb{P}\big[Z_{n-j}^{Q_j}>0\big].
\end{align}
First,  we use \eqref{eq: A_njc condioned S} to get that for each  ~$1\leq  j \leq  n$,
\begin{equation*}
	\mathbb{P}[A_{n,j}^c]\leq \mathbb{E}\big[S_j \big]\mathbb{P}\big[Z_{n-j}^{Q_j}>0\big].
\end{equation*}
The proof is completed by  using that $S_j$ is distributed as $[\dot{q}_j-1]$, and  the size-biased distribution of $Z_1^{Q_j}$ to obtain  
\begin{equation*}
	\begin{split}
		\mathbb{E}\big[S_j\big] &=  \mathbb{E}\big[\dot{Z}_1^{Q_{j-1}}-1\big] =\frac{ \mathbb{P}\big[Z_1^{Q_{j-1}}(Z_1^{Q_{j-1}}-1) \big] }{\mathbb{P}\big[Z_1^{Q_{j-1}}\big]}= \frac{f_j''(1)}{f_j'(1)}.
	\end{split}
\end{equation*}

\noindent \textbf{Step IV:} We will show that there exists a positive constant $C$ such that 
\begin{align}\label{lem:  Z_nj>0}
	\mathbb{P}\big[Z_{0}^{Q_n}>0\big] =1\qquad \mbox{and } \qquad \mathbb{P}\big[Z_{n-j}^{Q_j}>0\big] \leq    \frac{C}{\mu_j(\rho_{0,n}-\rho_{0,j})}, \quad \mbox{for } 0\leq j<n.
\end{align}	
We are going  to rely on the use of the shape functions defined in Subsection \ref{sec: shape}. By definition of the shape function $\varphi_{j,n}$  of $f_{j,n}$  evaluated at   $s=0$, we have that $\mathbb{P}\big[Z_{0}^{Q_n}>0\big] =1$, and for each $0\leq j<n$ 
\begin{equation*}
	\mathbb{P}\big[Z_{n-j}^{Q_j}>0\big]
	= 1-f_{j,n}(0)
	= \left(\frac{\mu_j}{\mu_n}+ \varphi_{j,n}(0) \right)^{-1}=\left(\frac{\mu_j}{\mu_n}+ \mu_j \sum_{l=j+1}^{n}\frac{\varphi_l(f_{l,n}(0))}{\mu_{l-1}}\right)^{-1}.
\end{equation*}
Now, we use the bound  \eqref{eq:bound varphi}  together with $\varphi_l(1)=\nu_l/2$ to get 

\begin{equation*}
	\mathbb{P}\big[Z_{n-j}^{Q_j}>0\big] \leq  \left(\mu_j \sum_{l=j+1}^{n}\frac{\varphi_l(f_{l,n}(0))}{\mu_{l-1}}\right)^{-1} \leq \left(\mu_j \sum_{l=j+1}^{n}\frac{C}{2}\frac{\nu_l}{\mu_{l-1}}\right)^{-1} =   \frac{C'}{\mu_j(\rho_{0,n}-\rho_{0,j})}.
\end{equation*}
Then, we obtained \eqref{lem:  Z_nj>0}.\\

\noindent \textbf{Step V:} 
Now, we put together all the ingredients to prove Theorem \ref{theo:speedYaglom} under Assumption \eqref{cond:thirdmoment}. The previous  assumption  implies \eqref{cond:Kersting}.  According with \cite[equation (3.2)]{kersting2021genealogical}, condition \eqref{cond:Kersting} implies that there exists  a constant $C_2>0$ such that 
\begin{equation}\label{eq:f''}
	f_j''(1) \leq C_2 (1+f_j'(1)) f_j'(1), \qquad j\geq 1.
\end{equation}
Recall that $b_n=\mathbb{E}[Z_n^Q\mid Z_n^Q>0]=\mu_n/\mathbb{P}[Z_n^Q>0].$ Then, we use \eqref{cond:thirdmoment} together with  \eqref{eq:meanYYe}, \eqref{lem:lemRnjprime},   \eqref{eq:StepII}, \eqref{lem:Anj}, \eqref{eq:f''} and the definition of $\nu_j$ to  obtain
\begin{equation*}
	\begin{split}
		d_W\left(\frac{Y_n}{b_n},\mathbf{e}\right) 
		&\leq \frac{2\mathbb{P}[Z_n^Q>0]}{\mu_n}\left(1
		+ \sum_{j=1}^{n} \mathbb{E}\left[\widetilde{R}_{n,j} \mathds{1}_{A_{n,j}^c} + R_{n,j}\mathds{1}_{A_{n,j}^c} \right]\right)\\
		&\leq \frac{2\mathbb{P}[Z_n^Q>0]}{\mu_n}\left(1
		+ \sum_{j=1}^{n} \mu_n\left( \frac{\nu_j}{\mu_{j-1}}\frac{f_j''(1)}{f_j'(1)}+
		\frac{f_j'''(1)+ f_j''(1)}{\mu_jf_j'(1)}\right)\mathbb{P}\big[Z_{n-j}^{Q_j}>0\big]\right)\\
		&\leq \frac{2\mathbb{P}[Z_n^Q>0]}{\mu_n}\left(1
		+ C \mu_n \sum_{j=1}^{n} \frac{\nu_j}{\mu_{j-1}} (1+f_j'(1)) \mathbb{P}\big[Z_{n-j}^{Q_j}>0\big]\right).
	\end{split}
\end{equation*}
Therefore, by using \eqref{lem:  Z_nj>0}, we have  
\begin{equation*}
	\begin{split}
		d_W\left(\frac{Y_n}{b_n},\mathbf{e}\right) 
		\leq& 
		\frac{C}{\mu_n\rho_{0,n}}
		+\frac{C}{\rho_{0,n}} \left(\sum_{j=1}^{n-1} 
		\frac{\nu_j}{\mu_{j-1}} \frac{(1+f_j'(1))}{\mu_j}
		\frac{1}{(\rho_{0,n}-\rho_{0,j})}
		+\frac{\nu_n}{\mu_{n-1}}(1+ f_n'(1))\right).
	\end{split}
\end{equation*}
Finally, if we define $r_n$ as in \eqref{eq:rn}  we get that Theorem \ref{theo:speedYaglom} is true.

\subsection{Proof of Theorem \ref{theo:rnestimation}}\label{sec:proofs2}
By Theorem \ref{theo:speedYaglom}, it suffices to bound the term $r_{n}$ given by \eqref{eq:rn}. Define 
\begin{align*}
	l_{n}
	&:=\sum_{j=1}^{n-1} 
	\frac{\nu_j}{\mu_{j-1}} \left(\frac{1}{\mu_j}+\frac{1}{\mu_{j-1}}\right)
	\frac{1}{(\rho_{0,n}-\rho_{0,j})}.
\end{align*}
Using the identity $\mu_{j}=\mu_{j-1}f_j^{\prime}(1)$, we deduce that
$$r_n=l_n+\frac{\nu_n}{\mu_{n-1}}(1+ f_n'(1)),$$
so the problem is reduced to estimating $l_n$. In order to do so, introduce the following quantities 
$$ h_k = \frac{1}{\mu_{k}}+\frac{1}{\mu_{k-1}}, \qquad g_1^{(n)}=0, \quad \text{and} \quad g_k^{(n)}= \sum_{j=1}^{k-1}\frac{\nu_j}{\mu_{j-1}} \frac{1}{(\rho_{0,n}- \rho_{0,j})}, \qquad \mbox{for } \ 2\le k\le n.$$
Then, 
\begin{align*}
	l_{n}
	&=\sum_{j=1}^{n-1}h_j(g_{j+1}^{(n)}-g_j^{(n)}).
\end{align*}
Applying  summation by parts, we get
\begin{equation*}
	\begin{split}
		l_n&=  h_{n-1} g_n^{(n)}  
		+ \sum_{k=2}^{n-1} \left(h_{k-1}-h_{k}\right)g_k^{(n)} = \left(\frac{1}{\mu_{n-1}}+\frac{1}{\mu_{n-2}}\right) g_n^{(n)} + 
		\sum_{k=2}^{n-1} \left(\frac{1}{\mu_{k-2}}- \frac{1}{\mu_k}\right)g_k^{(n)}.
	\end{split}
\end{equation*}

Recall that ~$\{\rho_{0,n}:   n\geq 0\}$ is an increasing sequence. Therefore, $g_k^{(n)}\leq g_k^{(k)},$ for each $2\le k\le n$, which gives
\begin{equation*}
	l_n  \leq f_{n-1}^{\prime}(1)(1+f_{n}^{\prime}(1))\frac{g_n^{(n)}}{\mu_{n}}  + \sum_{k=2}^{n-1}g_k^{(k)} \left|\frac{1}{\mu_{k-2}} - \frac{1}{\mu_{k}}\right|.
\end{equation*}
In order to bound the right-hand side, we will show that there is a constant $C>0$ such that for every $2\leq m\leq n$, 
\begin{align}\label{eq:logmurho}
	g_m^{(m)}
	=\sum_{j=1}^{m-1}\frac{\nu_j}{\mu_{j-1}} \frac{1}{(\rho_{0,m}- \rho_{0,j})}
	&\leq C (1+\mathfrak{M}_{n})^2 \left(\log(\mu_m \rho_{0,m})+\log\left(f_n'(1)\right)\right).
\end{align}
In this aim,  consider the following partition of  $[0,1]$,
$$P^{(n)}:=\{0=t_0^{(n)} < t_1^{(n)} < \cdots < t_n^{(n)}=1\} \qquad \text{where}\quad t_j^{(n)}:=\frac{\rho_{0,j}}{\rho_{0,n}}, \quad \text{for}\quad  j\in \{0, \dots, n\},$$
recalling that by definition $\rho_{0,0}=0$. 
Observe that 
\begin{equation}\label{eq:gnexpression}
	g_n^{(n)} 
	= \sum_{j=1}^{n-1} \frac{t_j^{(n)} - t_{j-1}^{(n)}}{1-t_j^{(n)}}   , \qquad n\geq 2,
\end{equation}
and $g_n^{(n)}$ can be regarded as the Riemman approximation of $\int_{0}^{1}  (1-x)^{-1}\mathrm{d}x$. Additionally, appealing to \eqref{eq:f''}, we get for all $j=1,\dots, n-2$, 
\begin{align*}
	t_j^{(n)} - t_{j-1}^{(n)}
	&=\frac{1}{\rho_{0,n}}\frac{\nu_{j}}{\mu_{j-1}}
	=\frac{1}{\rho_{0,n}}\frac{\nu_{j+1}}{\mu_{j}}\frac{\mu_{j}\nu_j}{\mu_{j-1}\nu_{j+1}}
	=(t_{j+1}^{(n)} - t_{j}^{(n)})f^{\prime}_j(1)\frac{\nu_j}{\nu_{j+1}}\\
	&=(t_{j+1}^{(n)} - t_{j}^{(n)})f^{\prime}_{j+1}(1)\frac{f_{j+1}'(1)}{f_j'(1)}\frac{f_{j}''(1)}{f_{j+1}''(1)} \leq C(t_{j+1}^{(n)} - t_{j}^{(n)})f^{'}_{j+1}(1)^2\frac{1+f_j'(1)}{f_{j+1}''(1)} \\ & \leq 
	C(1+\mathfrak{M}_{n})^3(t_{j+1}^{(n)} - t_{j}^{(n)}),
\end{align*}
where in the last inequality we have used that the sequence $\{f_n''(1): n\ge 1\}$ is bounded away from zero by assumption. 

We thus conclude that  
\begin{align*}
	\frac{t_j^{(n)} - t_{j-1}^{(n)}}{1-t_j^{(n)}} 
	&\leq C(1+\mathfrak{M}_{n})^3 \frac{t_{j+1}^{(n)} - t_{j}^{(n)}}{1-t_j^{(n)}}, \qquad \text{for all}\quad j=1,\dots, n-2.
\end{align*}
Moreover, by first separating the term $j=n-1$ in the right-hand side of \eqref{eq:gnexpression}, we can obtain, after a suitable change of indices, 
\begin{align*}
	g_n^{(n)} &\leq  C(1+\mathfrak{M}_{n})^3  \sum_{j=2}^{n-1} \frac{t_j^{(n)} - t_{j-1}^{(n)}}{1-t_{j-1}^{(n)}} + f_{n-1}'(1) \frac{f_n'(1)^2}{f_{n-1}'(1)^2} \frac{f_{n-1}''(1)}{f_n''(1)}
	\\ &\leq C(1+\mathfrak{M}_{n})^3 \left(\sum_{j=2}^{n-1} \frac{t_j^{(n)} - t_{j-1}^{(n)}}{1-t_{j-1}^{(n)}}+ 1\right),
\end{align*}
where in the last inequality we have again used \eqref{eq:f''} and  that the sequence $\{f_n''(1): n\ge 1\}$ is bounded away from zero. By the monotonicity of the function $x\mapsto (1-x)^{-1}$, we deduce 
\begin{eqnarray*}
	\sum_{j=2}^{n-1} \frac{t_j^{(n)} - t_{j-1}^{(n)}}{1-t_{j-1}^{(n)}}
	&\leq&  \int_{0}^{t_{n-1}^{(n)}}  \frac{\mathrm{d}x}{1-x} =  -\log(1-t_{n-1}^{(n)})  = -  \log\left(\frac{\nu_{n}/ \mu_{n-1}}{\rho_{0,n}}\right) \\ &=&  \log(\rho_{0,n} \mu_n ) +\log\left(\frac{f_n'(1)}{f_n''(1)}\right) \leq \log(\rho_{0,n} \mu_n ) +\log\left(f_n'(1)\right).
\end{eqnarray*}
Therefore, \eqref{eq:logmurho} holds. Finally, we are going to apply \eqref{eq:logmurho} to prove  Theorem \ref{theo:rnestimation}. 
In this case,
\begin{align*}
	r_n& =l_n+\frac{\nu_n}{\mu_{n-1}}(1+ f_n'(1))\\
	&\leq f_{n-1}^{\prime}(1)(1+f_{n}^{\prime}(1)) \frac{g_n^{(n)}}{\mu_n} + \sum_{k=2}^{n-1}g_k^{(k)} \left|\frac{1}{\mu_{k-2}} - \frac{1}{\mu_{k}}\right|+\frac{\nu_n}{\mu_{n-1}}(1+ f_n'(1))\\
	& \leq  \frac{C{(1+\mathfrak{M}_{n})^5}}{\mu_{n}}  \left(\log(\rho_{0,n} \mu_n ) +\log\left(f_n'(1)\right)\right) +  C(1+\mathfrak{M}_{n})^2 s_n +\frac{1}{\mu_{n}}(1+\mathfrak{M}_{n})^2,
\end{align*}
where $s_n$ is defined as in \eqref{eq:sn}.  This concludes the proof.

\section*{Acknowledgments}
We would like to thank an anonymous referee and the Associate Editor, who made a number of very helpful suggestions. We thank A. Röllin and E. Peköz for discussions about their results in \cite{pekoz2011}. We thank Jonathan Badin and Ya. M. Khusanbaev for pointing out an issue with Theorem 1.2. that we had in a previous version of this article.

\bibliographystyle{abbrv}
\bibliography{references2}

\end{document}